\documentclass[11pt,twoside,a4paper]{article}
\usepackage[T1]{fontenc}

\usepackage[australian]{babel}

\usepackage[driver=pdftex,margin=2.9cm,heightrounded=true,centering,includeheadfoot]{geometry}

\usepackage{hyperref}

\usepackage[intlimits,tbtags]{amsmath}
\usepackage{amssymb,amsfonts}
\usepackage{amsthm}
\usepackage[capitalise]{cleveref}

\allowdisplaybreaks
\numberwithin{equation}{section}

\usepackage[dvipsnames]{xcolor}
\definecolor{MyBlue}{cmyk}{1,0.13,0,0.63}
\definecolor{MyGreen}{cmyk}{0.91,0,0.88,0.52}
\definecolor{MyRed}{rgb}{.6,0,0}
\newcommand{\mylinkcolor}{MyBlue}
\newcommand{\mycitecolor}{MyGreen}
\newcommand{\myurlcolor}{MyRed}

\makeatletter
\def\@endtheorem{\endtrivlist}
\makeatother
\theoremstyle{plain}
\newtheorem{thm}{Theorem}[section]
\newtheorem{lem}[thm]{Lemma}
\newtheorem{prop}[thm]{Proposition}
\newtheorem{coro}[thm]{Corollary}
\theoremstyle{definition}
\newtheorem{defn}[thm]{Definition}
\newtheorem{remark}[thm]{Remark}

\newtheorem{notation}[thm]{Notation}
\newtheorem{assumption}[thm]{Assumption}

\renewcommand{\eqref}[1]{\labelcref{#1}}
\crefname{thm}{Theorem}{Theorems}
\crefname{lem}{Lemma}{Lemmas}
\crefname{prop}{Proposition}{Propositions}
\crefname{coro}{Corollary}{Corollaries}
\crefname{defn}{Definition}{Definitions}
\crefname{example}{Example}{Examples}
\crefname{remark}{Remark}{Remarks}

\hypersetup{%
  bookmarksnumbered=true,bookmarksopen=false,%
  plainpages=false,
  linktocpage=true,%
  colorlinks=true,breaklinks=true,%
  linkcolor=\mylinkcolor,citecolor=\mycitecolor,urlcolor=\myurlcolor,%
  pdfpagelayout=OneColumn,%
  pageanchor=true,%
}

\makeatletter
\def\thm@space@setup{%
  \thm@preskip=4pt plus 2pt minus 2pt
  \thm@postskip=\thm@preskip
}
\renewenvironment{proof}[1][\proofname]{\par
  \pushQED{\qed}%
  \normalfont \topsep4\p@\relax 
  \trivlist
  \item[\hskip\labelsep
        \itshape
    #1\@addpunct{.}]\ignorespaces
}{%
  \popQED\endtrivlist\@endpefalse
}
\makeatother

\usepackage[shortlabels]{enumitem}
\setlist{topsep=4pt plus 2pt minus 2pt,partopsep=0pt,itemsep=2pt plus 2pt minus 2pt,parsep=0.5\parskip}
\setenumerate[1]{label=(\arabic*),font=\upshape}

\usepackage{fancyhdr}
\newcommand{\MR}[1]{}

\let\OLDthebibliography\thebibliography
\renewcommand\thebibliography[1]{
  \addcontentsline{toc}{section}{\refname}
  \OLDthebibliography{#1}
  \setlength{\parskip}{0pt}
  \setlength{\itemsep}{0pt plus 0.3ex}
}

\newlist{enumlocal}{enumerate}{1}
\setlist[enumlocal]{label={(L\arabic*)},resume=foo}

\newcommand{\N}{\mathbb{N}}
\newcommand{\R}{\mathbb{R}}
\newcommand{\C}{\mathbb{C}}
\newcommand{\Z}{\mathbb{Z}}
\newcommand{\A}{\mathcal{A}}

\newcommand{\D}{\mathcal{D}}
\newcommand{\B}{\mathcal{B}}
\newcommand{\mO}{\mathcal{O}}
\newcommand{\mF}{\mathcal{F}}
\newcommand{\mS}{\mathcal{S}}
\newcommand{\mT}{\mathcal{T}}
\newcommand{\E}{\mathcal{E}}

\DeclareMathOperator{\Dom}{Dom}
\DeclareMathOperator{\Ran}{Ran}

\DeclareMathOperator{\Id}{Id}
\DeclareMathOperator{\End}{End}
\DeclareMathOperator{\Hom}{Hom}
\DeclareMathOperator{\supp}{supp}
\DeclareMathOperator{\Lip}{Lip}
\DeclareMathOperator{\ev}{ev}

\renewcommand{\Re}{\mathop{\textnormal{Re}}}

\renewcommand{\bar}[1]{\overline{#1}}

\newcommand{\K}{K}
\newcommand{\KK}{K\!K}

\newcommand{\til}[1]{\widetilde{#1}}
\newcommand{\hotimes}{\mathbin{\hat\otimes}}
\newcommand{\hot}{\hotimes}
\newcommand{\la}{\langle}
\newcommand{\ra}{\rangle}
\newcommand{\into}{\hookrightarrow}
\newcommand{\mvert}{\,|\,}
\newcommand{\bigmvert}{\,\big|\,}

\newcommand{\mattwo}[4]{
  \begin{pmatrix}#1&#2\\ #3&#4\end{pmatrix}
}

\makeatletter
\DeclareFontFamily{OMX}{MnSymbolE}{}
\DeclareSymbolFont{MnLargeSymbols}{OMX}{MnSymbolE}{m}{n}
\SetSymbolFont{MnLargeSymbols}{bold}{OMX}{MnSymbolE}{b}{n}
\DeclareFontShape{OMX}{MnSymbolE}{m}{n}{
    <-6>  MnSymbolE5
   <6-7>  MnSymbolE6
   <7-8>  MnSymbolE7
   <8-9>  MnSymbolE8
   <9-10> MnSymbolE9
  <10-12> MnSymbolE10
  <12->   MnSymbolE12
}{}
\DeclareFontShape{OMX}{MnSymbolE}{b}{n}{
    <-6>  MnSymbolE-Bold5
   <6-7>  MnSymbolE-Bold6
   <7-8>  MnSymbolE-Bold7
   <8-9>  MnSymbolE-Bold8
   <9-10> MnSymbolE-Bold9
  <10-12> MnSymbolE-Bold10
  <12->   MnSymbolE-Bold12
}{}

\let\llangle\@undefined
\let\rrangle\@undefined
\DeclareMathDelimiter{\llangle}{\mathopen}%
                     {MnLargeSymbols}{'164}{MnLargeSymbols}{'164}
\DeclareMathDelimiter{\rrangle}{\mathclose}%
                     {MnLargeSymbols}{'171}{MnLargeSymbols}{'171}
\makeatother

\newcommand{\lla}{\llangle}
\newcommand{\rra}{\rrangle}

\def\MyTitle{Localisations of half-closed modules \\ and the unbounded Kasparov product}
\def\MyShortTitle{Localisations and the Kasparov product}
\title{\MyTitle}
\author{
Koen van den Dungen%
\footnote{Email: \texttt{kdungen@uni-bonn.de}}
\\[2mm]
{\small Mathematisches Institut}, 
{\small Universit\"at Bonn}\\
{\small Endenicher Allee 60, D-53115 Bonn}
}

\date{}

\begin{document}

\maketitle

\begin{abstract}
\noindent
In the context of the Kasparov product in unbounded $\KK$-theory, a well-known theorem by Kucerovsky provides sufficient conditions for an unbounded Kasparov module to represent the (internal) Kasparov product of two other unbounded Kasparov modules. 
In this article, we discuss several improved and generalised variants of Kucerovsky's theorem. 
First, we provide a generalisation which relaxes the positivity condition, by replacing the lower bound by a \emph{relative} lower bound. 
Second, we also discuss Kucerovsky's theorem in the context of half-closed modules, which generalise unbounded Kasparov modules to \emph{symmetric} (rather than self-adjoint) operators. 
In order to deal with the positivity condition for such non-self-adjoint operators, we introduce a fairly general localisation procedure, which (using a suitable approximate unit) provides a `localised representative' for the $\KK$-class of a half-closed module. 
Using this localisation procedure, we then prove several variants of Kucerovsky's theorem for half-closed modules. 
A distinct advantage of the localised approach, also in the special case of self-adjoint operators (i.e., for unbounded Kasparov modules), is that the (global) positivity condition in Kucerovsky's original theorem is replaced by a (less restrictive) `local' positivity condition, which is closer in spirit to the well-known Connes-Skandalis theorem in the bounded picture of $\KK$-theory. 

\vspace{\baselineskip}
\noindent
\emph{Keywords}: Unbounded $\KK$-theory; the Kasparov product; symmetric operators.

\noindent
\emph{Mathematics Subject Classification 2010}: 
19K35. 
\end{abstract}

\newpage 

\section{Introduction}

Let $A$, $B$, and $C$ be $C^*$-algebras. 
Kasparov's $\KK$-theory \cite{Kas80b} provides the abelian group $\KK(A,B)$ of homotopy equivalence classes of (bounded) Kasparov $A$-$B$-modules. 
One of the main features of $\KK$-theory is the existence of an associative bilinear pairing called the Kasparov product:
\[
\KK(A,B) \times \KK(B,C) \to \KK(A,C) .
\]
However, it is not possible to compute the Kasparov product explicitly in general. 
An important improvement was made by Connes and Skandalis \cite{CS84}, who provided sufficient conditions (the so-called connection condition and positivity condition) which ensure that a certain Kasparov module represents the Kasparov product of two other given Kasparov modules. 

It was shown by Baaj and Julg \cite{BJ83} that elements in $\KK$-theory can also be represented by \emph{unbounded} Kasparov modules. 
Many examples of elements in $\KK$-theory are constructed from geometric situations and are most naturally represented by unbounded Kasparov modules. 
For example, the fundamental class in the $\K$-homology of a spin manifold is naturally represented by the Dirac operator (viewed as an \emph{unbounded} operator on the Hilbert space of spinors). 
A distinct advantage of the unbounded picture is that the Kasparov product is often easier to compute; in fact, under suitable assumptions, the unbounded Kasparov product can be explicitly \emph{constructed} \cite{Mes14,KL13,BMS16}. 

However, the unbounded Kasparov modules introduced by Baaj and Julg require the unbounded operators to be \emph{self-adjoint}. 
This self-adjointness is rather natural in the case of unital $C^*$-algebras (e.g.\ for compact manifolds), but imposes a completeness condition in the non-unital case (e.g.\ for non-compact manifolds) \cite{MR16}. 
Thus, unfortunately, the Baaj-Julg framework does not include typical examples such as Dirac-type operators on non-compact, \emph{incomplete}, Riemannian spin manifolds (which in general are only symmetric but not self-adjoint). 
Nevertheless, it was shown by Baum-Douglas-Taylor \cite{BDT89} that any elliptic symmetric first-order differential operator $\D$ on a Riemannian manifold $M$ yields a class $[\D] := [F_\D]$ in the $\K$-homology of $M$, given by the \emph{bounded transform} $F_\D := \D(1+\D^*\D)^{-\frac12}$. 
Hilsum \cite{Hil10} later provided an abstract definition of \emph{half-closed modules}, which generalise unbounded Kasparov modules by replacing the self-adjointness condition by a more flexible symmetry condition. Moreover, Hilsum proved that the bounded transform of a half-closed module still yields a well-defined class in $\KK$-theory. 
Hilsum's framework encompasses the symmetric elliptic operators of Baum-Douglas-Taylor as well as the \emph{vertically elliptic} operators on submersions of (possibly incomplete) smooth manifolds studied in \cite{KS18,KS20,vdD20_Kasp_open}. 

An important step in the development of the unbounded Kasparov product was provided by Kucerovsky's theorem \cite{Kuc97}, which translated the Connes-Skandalis result to the unbounded framework of Baaj and Julg. 
Thus, given two unbounded Kasparov modules $(\A,(E_1)_B,\D_1)$ and $(\B,(E_2)_C,\D_2)$, Kucerovsky provided sufficient conditions for another unbounded Kasparov module $(\A,E_C,\D)$ to represent the (internal) Kasparov product. 
However, while the positivity condition of Connes-Skandalis is a \emph{local} condition, the positivity condition of Kucerovsky is a \emph{global} condition (and is therefore more restrictive). 
Moreover, Kucerovsky's theorem only applies to the self-adjoint (complete) case. 

In this article, we discuss several improved and generalised variants of Kucerovsky's theorem. 
In particular, we will make the following two improvements to the so-called positivity condition:
\begin{enumerate}[label=(\Roman*)]
\item \label{improvement:first-order}
we replace the lower bound by a \emph{relative} lower bound with respect to the operator representing the Kasparov product;
\item \label{improvement:local}
we ensure that the positivity condition only needs to be checked \emph{locally} (rather than globally), which is closer in spirit to the positivity condition of Connes-Skandalis. 
\end{enumerate}
Our `local' approach to the positivity condition in particular also allows us to consider the non-self-adjoint (incomplete) case of \emph{half-closed modules}. 
In fact, it was one of our main goals in this article to enhance our understanding of the unbounded Kasparov product of two half-closed modules. 
The main motivating examples come from vertically elliptic operators on submersions of open manifolds, for which the Kasparov product was already described in \cite{vdD20_Kasp_open}. 
We will show that similar results can be obtained in the context of \emph{noncommutative} $C^*$-algebras as well. 
Specifically, we prove several localised versions of Kucerovsky's theorem for half-closed modules. 

We emphasise here that our localised approach can also be of advantage in the self-adjoint (complete) case of unbounded Kasparov modules, since the `local' positivity condition is more flexible and therefore more broadly applicable. 
Indeed, consider for instance the simple example of the Riemannian submersion of Euclidean spaces $\R^2 \to \R$, $(x,y) \mapsto x$. 
Then the vertical operator $\D_1 := -i\partial_y + \sin(e^x)$ defines a class in $\KK^1(C_0(\R^2),C_0(\R))$, and the operator $\D_2 := -i\partial_x$ defines a class in $\KK^1(C_0(\R),\C)$. 
Their Kasparov product is the fundamental class in the $\K$-homology of $C_0(\R^2)$, represented by the standard Dirac operator on $\R^2$, which can be checked by simply ignoring the bounded perturbation $\sin(e^x)$ in $\D_1$. 
However, this fact cannot be checked \emph{directly} (without ignoring the perturbation) by an application of Kucerovsky's theorem. Indeed, the commutator 
$
[\D_1,\D_2] = i e^x \cos(e^x) 
$
is not bounded (nor relatively bounded by $\D_1$ and/or $\D_2$), and Kucerovsky's positivity condition fails. 
This illustrates the problem with the uniformity of the estimate in Kucerovsky's positivity condition. 
But, as the commutator $[\D_1,\D_2]$ is certainly \emph{locally} bounded, our \emph{local} positivity condition (see \cref{defn:local_positivity_condition}) is satisfied without needing to change the operator $\D_1$. 

We point out that Kaad and Van Suijlekom have also proved a variant of Kucerovsky's theorem for half-closed modules \cite[Theorem 6.10]{KS19}, which partially addresses item \ref{improvement:local}. 
However, in their work, if the left module is essential, then the required assumptions in \cite[Theorem 6.10]{KS19} imply that $\D_1$ is self-adjoint, so that the left module is in fact an unbounded Kasparov module (see \cite[Remark 6.5]{KS19}). 
In the context of differential operators on Riemannian submersions, to obtain the self-adjointness of $\D_1$ one needs to assume for instance (cf.\ \cite{KS20}) that the submersion is \emph{proper} (which means that the fibres are compact). 
It was one of the main goals of this article to obtain a version of Kucerovsky's theorem which includes the case of submersions with non-compact and incomplete fibres as well, so that in particular the results in \cite{vdD20_Kasp_open} are recovered as a special case. 

Let us summarise the contents and main results of this article. 
We start \cref{sec:prelim} by recalling some preliminaries on $\KK$-theory and the Kasparov product, and then review Hilsum's framework of half-closed modules in \cref{sec:half-closed}. 
In \cref{sec:reg_symm} we collect a few useful lemmas regarding regular symmetric operators on Hilbert modules. 
In \cref{sec:connection} we discuss the connection condition, which can be dealt with in the non-self-adjoint case in the same way as in the self-adjoint case, without additional difficulty. 
We will show that the connection condition for half-closed modules implies the connection condition of Connes-Skandalis. 

In \cref{sec:Kasp_prod_revisited}, we first restrict our attention to the self-adjoint case of unbounded Kasparov modules, and we show that the lower bound in Kucerovsky's positivity condition can be replaced by a \emph{relative} lower bound (as mentioned in \ref{improvement:first-order} above). 
To motivate this generalisation, let us first compare to the aforementioned Connes-Skandalis Theorem (see \cref{thm:Connes-Skandalis}). 
Consider three (bounded) Kasparov modules $(\A,(E_1)_B,F_1)$, $(\B,(E_2)_C,F_2)$, and $(\A,E_B,F)$ (where $E=E_1\hot_BE_2$). 
Roughly speaking, the positivity condition of Connes-Skandalis requires the (graded) commutator $[F,F_1\hot1]$ to be positive modulo compact operators. If we think of this in terms of pseudodifferential operators, this means there exists a positive zeroth-order pseudodifferential operator with the same principal symbol as $[F,F_1\hot1]$. 
On the other hand, Kucerovsky's positivity condition (roughly speaking) requires the (graded) commutator $[\D,\D_1\hot1]$ to be positive modulo bounded operators (i.e., modulo `zeroth-order' operators). 
This seems somewhat unnatural, since if we think of this in terms of differential operators, it means that the positivity condition depends not only on the (second-order) principal symbol of $[\D,\D_1\hot1]$, but also on the first-order part of the symbol (even though the $\KK$-classes of $\D$ and $\D_1$ are determined by their principal symbols alone). 
A more natural condition would therefore be to require $[\D,\D_1\hot1]$ to be positive modulo `first-order' operators, meaning that there exists a positive second-order differential operator with the same principal symbol as $[\D,\D_1\hot1]$. 
We will prove in \cref{thm:Kucerovsky_revisited} that Kucerovsky's theorem can indeed be generalised to this more natural positivity condition. 

The unbounded Kasparov product has been studied in \cite{LM19} in the context of weakly anti-commuting operators. 
As a corollary to \cref{thm:Kucerovsky_revisited}, we obtain here an alternative proof of \cite[Theorem 7.4]{LM19}. 

In the remainder of this article, we consider the symmetric (non-self-adjoint) case of half-closed modules. 
In \cref{sec:localisations} we describe our localisation procedure for half-closed modules, and provide the construction of a \emph{localised representative} $\til F_\D$ for a half-closed module $(\A,E_B,\D)$. 
This construction was first given by Higson \cite{Hig89pre} (see also \cite[\S10.8]{Higson-Roe00}) for the case of elliptic symmetric first-order differential operators, and was generalised by the author \cite{vdD20_Kasp_open} to the case of \emph{vertically elliptic} symmetric first-order differential operators on submersions of open manifolds. 
We will show in \cref{sec:local_representative} that Higson's construction of a localised representative can be generalised further to the abstract noncommutative setting of a half-closed module $(\A,E_B,\D)$. The construction is based on the assumption that there exists an almost idempotent approximate unit $\{u_n\}_{n\in\N}$ in the dense $*$-subalgebra $\A\subset A$. 
Such approximate units always exist in any $\sigma$-unital $C^*$-algebra $A$; the only additional assumption here is that it must lie in the subalgebra $\A$. 
Under this assumption, we show how to construct a localised representative for the $\KK$-class of a half-closed module. 

A detailed treatment of the local version of the positivity condition is given in \cref{sec:local_positivity}. 
The main technical result consists of showing that our `local' positivity condition for half-closed modules implies that the positivity condition of Connes-Skandalis is satisfied `locally'. The proof relies on the following two features of the localised representative. 
First, it allows us to work `locally' with self-adjoint (rather than only symmetric) operators. 
Second, each localised term can be rescaled independently (which is crucial in order to obtain a uniform constant in the positivity condition). 

In \cref{sec:Kasp_prod}, we then finally provide several localised versions of Kucerovsky's theorem for half-closed modules, 
which provide sufficient conditions for a half-closed module $(\A,E_C,\D)$ to represent the (internal) Kasparov product of two half-closed modules $(\A,(E_1)_B,\D_1)$ and $(\B,(E_2)_C,\D_2)$. 
All these versions require the existence of a suitable approximate unit $\{u_n\}\subset\A$ as in \cref{sec:local_representative}. 
Our first main result (\cref{thm:Kucerovsky_half-closed_localised}) requires (in addition to a domain condition) the following local version of the positivity condition: 
for each $n\in\N$ there exists $c_n\in[0,\infty)$ such that for all $\psi\in\Dom(\D u_n) \cap \Ran(u_n)$ we have 
\begin{align}
\label{eq:positivity_localised_intro}
\big\la u_n(\D_1\hot1)\psi \bigmvert \D u_n\psi \big\ra + \big\la \D u_n\psi \bigmvert u_n(\D_1\hot1)\psi \big\ra 
\geq - c_n \big\la \psi \bigmvert (1+(u_n\D u_n)^2)^{\frac12} \psi \big\ra . 
\end{align}
Using the results from \cref{sec:local_positivity}, \cref{eq:positivity_localised_intro} ensures that the positivity condition of Connes-Skandalis is satisfied `locally'. 
The construction of the localised representative $\til F_{\D_1}$ using a `partition of unity' then allows us to prove that the positivity condition is in fact satisfied globally. 

Roughly speaking, condition \eqref{eq:positivity_localised_intro} requires that the anti-commutator $\big[ \D_1\hot1 , u_n\D u_n \big]$ is positive modulo `first-order operators' (where `first-order' is defined relative to $u_n\D u_n$). However, to obtain a better analogy to the Connes-Skandalis positivity condition, it would be more natural to require $u_n [\D_1\hot1,\D] u_n$ to be positive modulo `first-order operators' (where `first-order' should be defined relative to $\D$). 
Unless $\D_1$ commutes with $u_n$ (as in \cite{KS19}), this additional step is non-trivial (indeed, although we know that $[\D_1,u_n]$ is bounded, this alone does not guarantee that we may consider $[\D_1\hot1,u_n]\D u_n$ to be `first-order'). 
In \cref{sec:local_Kucerovsky}, we will consider two possible sufficient conditions which allow us to make this additional step. 

The first sufficient condition assumes, instead of \eqref{eq:positivity_localised_intro}, 
a \emph{strong} local positivity condition, which requires that (for each $n\in\N$) there exist $\nu_n\in(0,\infty)$ and $c_n\in[0,\infty)$ such that for all $\psi\in\Dom(\D)$ we have 
\begin{multline*}
\big\la (\D_1\hot1)u_n\psi \bigmvert \D u_n\psi \big\ra + \big\la \D u_n\psi \bigmvert (\D_1\hot1)u_n\psi \big\ra \\*
\geq \nu_n \big\la (\D_1\hot1)u_n\psi \bigmvert (\D_1\hot1)u_n\psi \big\ra - c_n \big\la u_n\psi \bigmvert (1+\D^*\D)^{\frac12} u_n\psi \big\ra . 
\end{multline*}

The second sufficient condition assumes, instead of \eqref{eq:positivity_localised_intro}, a \emph{local positivity condition}, which requires simply that (for each $n\in\N$) there exists $c_n\in[0,\infty)$ such that for all $\psi\in\Dom(\D)$ we have 
\begin{align*}
\big\la (\D_1\hot1)u_n\psi \bigmvert \D u_n\psi \big\ra + \big\la \D u_n\psi \bigmvert (\D_1\hot1)u_n\psi \big\ra \geq - c_n \big\la u_n\psi \bigmvert (1+\D^*\D)^{\frac12} u_n\psi \big\ra , 
\end{align*}
along with a \emph{`differentiability' condition}, which requires that the operator $u_n[\D,u_n]u_{n+2}$ maps $\Dom(\D u_{n+2}^2)$ to $\Dom(\D_1\hot1)$. 

We note that the latter `differentiability' condition is quite naturally satisfied in the context of first-order differential operators on smooth manifolds, when $u_n$ are compactly supported smooth functions and $\D$ is elliptic (as in \cite{vdD20_Kasp_open}). 
In \cref{sec:construction} we will show that the \emph{strong} local positivity condition is in fact fairly natural in the constructive approach to the unbounded Kasparov product.

\subsection{Acknowledgements}

This article is a continuation of \cite{vdD20_Kasp_open}, and the work on both these articles was initiated during a short visit to the Radboud University Nijmegen in late 2017, which was funded by the COST Action MP1405 QSPACE, supported by COST (European Cooperation in Science and Technology). 
The author thanks Walter van Suijlekom for his hospitality during this visit, and for interesting discussions. 
Thanks also to Bram Mesland and Matthias Lesch for interesting discussions.

\subsection{Notation}

Let $A$ and $B$ denote $\sigma$-unital $\Z_2$-graded $C^*$-algebras. 
By an approximate unit for $A$ we will always mean an even, positive, increasing, and contractive approximate unit for the $C^{*}$-algebra $A$. 
Let $E$ be a $\Z_2$-graded Hilbert module over $B$ (for an introduction to Hilbert modules and further details, see for instance \cite{Lance95,Blackadar98}). 
For $\xi\in E$, we consider the short-hand notation 
\[
\lla \xi \rra := \la \xi | \xi \ra ,
\]
where $\la \cdot | \cdot \ra$ denotes the $B$-valued inner product on $E$. 
We denote the set of adjointable operators on $E$ as $\End_B(E)$, and the subset of compact endomorphisms as $\End_B^0(E)$. 
For any operator $T$ on $E$, we write $\deg T=0$ if $T$ is even, and $\deg T=1$ if $T$ is odd. 
The graded commutator $[\cdot,\cdot]$ is defined (on homogeneous operators) by $[S,T] := ST - (-1)^{\deg S\cdot\deg T} TS$. 

For any $S,T\in\End_B(E)$ we will write $S \sim T$ if $S-T\in\End_B^0(E)$. Similarly, for self-adjoint $S,T$ we will write $S\gtrsim T$ if $S-T \sim P$ for some positive $P\in\End_B(E)$; in this case we will say that $S-T$ is \emph{positive modulo compact operators}. 

Given a $*$-homomorphism $A\to\End_B(E)$, an operator $T\in\End_B(E)$ is called \emph{locally compact} if $aT$ is compact for every $a\in A$. 

Given any regular operator $\D$, we define the \emph{bounded transform} $F_\D := \D (1+\D^*\D)^{-\frac12}$. 
The graph inner product of $\D$ is given for $\psi\in\Dom\D$ by 
\[
\la\psi|\psi\ra_\D := \la\psi|\psi\ra + \la\D\psi|\D\psi\ra , 
\]
and the corresponding graph norm is given by $\|\psi\|_\D^2 := \|\la\psi|\psi\ra_\D\|$.

\section{Preliminaries on \texorpdfstring{$\KK$}{KK}-theory}
\label{sec:prelim}

Kasparov \cite{Kas80b} defined the abelian group $\KK(A,B)$ as a set of homotopy equivalence classes of Kasparov $A$-$B$-modules. 
We start by briefly recalling the main definitions; for more details we refer to e.g.\ \cite[\S17]{Blackadar98}. 

\begin{defn}
\label{defn:Kasp_mod}
A (bounded) \emph{Kasparov $A$-$B$-module} $(A,{}_{\pi}E_B,F)$ is given by a $\Z_2$-graded countably generated right Hilbert $B$-module $E$, a ($\Z_2$-graded) $*$-homomorphism $\pi\colon A\to\End_B(E)$, and an odd adjointable endomorphism $F\in\End_B(E)$ such that for all $a\in A$: 
\[\pi(a)(F-F^*), \quad [F,\pi(a)], \quad \pi(a)(F^2-1)\in\End^{0}_{B}(E).\] 

Two Kasparov $A$-$B$-modules $(A,{}_{\pi_0}{E_0}_B,F_0)$ and $(A,{}_{\pi_1}{E_1}_B,F_1)$ are called \emph{unitarily equivalent} (denoted with $\simeq$) if there exists an even unitary in $\Hom_B(E_0,E_1)$ intertwining the $\pi_j$ and $F_j$ (for $j=0,1$). 

A \emph{homotopy} between $(A,{}_{\pi_0}{E_0}_B,F_0)$ and $(A,{}_{\pi_1}{E_1}_B,F_1)$ is given by a Kasparov $A$-$C([0,1],B)$-module $(A,{}_{\til\pi}{\til E}_{C([0,1],B)},\til F)$ such that (for $j=0,1$)
\[
\ev_j(A,{}_{\til\pi}{\til E}_{C([0,1],B)},\til F) \simeq (A,{}_{\pi_j}{E_j}_B,F_j) .
\]
Here $\simeq$ denotes unitary equivalence, and $\ev_t(A,{}_{\til\pi}{\til E}_{C([0,1],B)},\til F) := (A,{}_{\til\pi\hot1}{\til E\hot_{\rho_t}B}_B,\til F\hot1)$, where the $*$-homomorphism $\rho_t\colon C([0,1],B) \to B$ is given by $\rho_t(b) := b(t)$. 

A homotopy $(A,{}_{\til\pi}{\til E}_{C([0,1],B)},\til F)$ is called an \emph{operator-homotopy} if there exists a Hilbert $B$-module $E$ with a representation $\pi\colon A\to\End_B(E)$ such that $\til E$ equals the Hilbert $C([0,1],B)$-module $C([0,1],E)$ with the natural representation $\til\pi$ of $A$ on $C([0,1],E)$ induced from $\pi$, and if $\til F$ is given by a \emph{norm}-continuous family $\{F_t\}_{t\in[0,1]}$. 
A module $(\pi,E,F)$ is called \emph{degenerate} if $\pi(a)(F-F^*) = [F,\pi(a)] = \pi(a)(F^2-1) = 0$ for all $a\in A$. 

The $\KK$-theory $\KK(A,B)$ of $A$ and $B$ is defined as the set of homotopy equivalence classes of (bounded) Kasparov $A$-$B$-modules. 
Since homotopy equivalence respects direct sums, the direct sum of Kasparov $A$-$B$-modules induces a (commutative and associative) binary operation (`addition') on the elements of $\KK(A,B)$ such that $\KK(A,B)$ is in fact an abelian group \cite[\S4, Theorem 1]{Kas80b}. 

If no confusion arises, we often simply write $(A,E_B,F)$ instead of $(A,{}_{\pi}E_B,F)$, and its class in $\KK$-theory is simply denoted by $[F] \in \KK(A,B)$. 
\end{defn}

\subsection{The Kasparov product}

Let $A$ be a ($\Z_2$-graded) separable $C^*$-algebra, and let $B$ and $C$ be ($\Z_2$-graded) $\sigma$-unital $C^*$-algebras. It was shown by Kasparov \cite[\S4, Theorem 4]{Kas80b} that there exists an associative bilinear pairing, called the (internal) \emph{Kasparov product}:
\[
\KK(A,B) \times \KK(B,C) \to \KK(A,C) .
\]
Given two $\KK$-classes $[F_1] \in \KK(A,B)$ and $[F_2] \in \KK(B,C)$, the Kasparov product is denoted by $[F_1]\otimes_B[F_2]$. 

An important improvement 
was provided by Connes and Skandalis \cite{CS84}, who gave sufficient conditions which allow to check whether a given Kasparov module represents the Kasparov product. 
For convenience, let us first introduce some notation. Given a Hilbert $B$-module $E_1$ and a Hilbert $C$-module $E_2$ with a $*$-homomorphism $B\to\End_C(E_2)$, we consider the (graded) internal tensor product $E := E_1\hot_BE_2$. For any $\psi\in E_1$, we define the operator $T_\psi \colon E_2 \to E$ as $T_\psi \eta = \psi\hot\eta$ for any $\eta\in E_2$. The operator $T_\psi$ is adjointable, and its adjoint $T_\psi^*\colon E\to E_2$ is given by $T_\psi^* (\xi\otimes\eta) = \la\psi|\xi\ra\cdot\eta$. 
Furthermore, we also introduce the operator $\til T_\psi$ on the Hilbert $C$-module $E\oplus E_2$ given by 
\[
\til T_\psi := \mattwo{0}{T_\psi}{T_\psi^*}{0} . 
\]
We cite here a slightly more general version of the theorem by Connes and Skandalis. 
First, as explained by Kucerovsky \cite[Proposition 5]{Kuc97}, it suffices to check the connection condition for $\psi\in \pi_1(A)E_1$ (rather than all $\psi\in E_1$). 
Second, as described in the comments following \cite[Definition 18.4.1]{Blackadar98}, the positivity condition in fact only requires a lower bound greater than $-2$. 

\begin{thm}[{\cite[Theorem A.3]{CS84}}]
\label{thm:Connes-Skandalis}
Consider two Kasparov modules $(A,{}_{\pi_1}(E_1)_B,F_1)$ and $(B,{}_{\pi_2}(E_2)_C,F_2)$, and consider the Hilbert $C$-module $E := E_1\hot_B E_2$ and the $*$-homo\-mor\-phism $\pi := \pi_1\hot1 \colon A \to \End_C(E)$. 
Suppose that $(A,{}_\pi E_C,F)$ is a Kasparov module such that the following two conditions hold:
\begin{description}
\item[Connection condition:] for any $\psi\in \pi_1(A)\cdot E_1$, the graded commutator $[F\oplus F_2,\til T_\psi]$ is compact on $E\oplus E_2$; 
\item[Positivity condition:] there exists a $0\leq\kappa<2$ such that for all $a\in A$ we have that $\pi(a) [F_1\hot1,F] \pi(a^*) + \kappa\pi(aa^*)$ is positive modulo compact operators on $E$. 
\end{description}
Then $(A,{}_\pi E_C,F)$ represents the Kasparov product of $(A,{}_{\pi_1}{E_1}_B,F_1)$ and $(B,{}_{\pi_2}{E_2}_C,F_2)$:
\[
[F] = [F_1] \hot_B [F_2] \in \KK(A,C) .
\]
Moreover, an operator $F$ with the above properties always exists and is unique up to operator-homotopy. 
\end{thm}

\subsection{Half-closed modules}
\label{sec:half-closed}

Let $A$ and $B$ be $\Z_2$-graded $\sigma$-unital $C^*$-algebras, and let $E$ be a $\Z_2$-graded countably generated Hilbert $B$-module.  
For any densely defined operator $\D$ on $E$, we consider the following subspaces of $\End_B(E)$:
\begin{align*}
\Lip(\D) &:= \big\{ T\in\End_B(E) : T\cdot\Dom\D\subset\Dom\D , \text{ and $[\D,T]$ is bounded on $\Dom\D$} \big\} , \\
\Lip^*(\D) &:= \big\{ T\in\Lip(\D) : T\cdot\Dom\D^*\subset\Dom\D \big\} .
\end{align*}
If $\D^*$ is also densely defined, we note that $T\in\Lip(\D)$ implies $T^*\in\Lip(\D^*)$, and then $-[\D,T]^*$ equals the closure of $[\D^*,T^*]$ (see \cite[Lemma 2.1]{Hil10}).
Moreover, if $\D$ and $T$ are symmetric, then we have $\bar{[\D^*,T]} = \bar{[\D,T]}$. 

\begin{defn}[{\cite[\S2]{Hil10}}]
\label{defn:Hilsum}
A {half-closed $A$-$B$-module} $(\A,E_B,\D)$ is given by a $\Z_2$-graded countably generated Hilbert $B$-bimodule $E$, an odd regular symmetric operator $\D$ on $E$, a $*$-homomorphism $A\to\End_B(E)$, and a dense $*$-subalgebra $\A\subset A$ such that 
\begin{enumerate}
\item $\A\subset\Lip^*(\D)$; 
\item $(1+\D^*\D)^{-1}$ is locally compact. 
\end{enumerate}
If furthermore $\D$ is self-adjoint, then $(\A,E_B,\D)$ is called an \emph{unbounded Kasparov $A$-$B$-module}. 
\end{defn}

Unbounded Kasparov modules were first introduced by Baaj and Julg \cite{BJ83}, who proved that their bounded transforms yield Kasparov modules. This statement was generalised to half-closed modules by Hilsum. 

\begin{thm}[{\cite[Theorem 3.2]{Hil10}}]
\label{thm:Hilsum}
Let $(\A,E_B,\D)$ be a half-closed $A$-$B$-module, and consider a closed extension $\D\subset\hat\D\subset\D^*$. Then the bounded transform $F_{\hat\D} = \hat\D(1+\hat\D^*\hat\D)^{-\frac12}$ yields a Kasparov $A$-$B$-module $(A,E_B,F_{\hat\D})$, and its class is independent of the choice of the extension $\hat\D$. 
\end{thm}

The following theorem by Kucerovsky provides an analogous version of the Connes-Skandalis result (\cref{thm:Connes-Skandalis}) for the Kasparov product of \emph{unbounded Kasparov modules}. (Note that we have simplified the domain condition using \cite[Lemma 10(i)]{Kuc97}.) 

\begin{thm}[{\cite[Theorem 13]{Kuc97}}]
\label{thm:Kucerovsky}
Let $(\A, {}_{\pi_1}(E_1)_B,\D_1)$ and $(\B,{}_{\pi_2}(E_2)_C,\D_2)$ be unbounded Kasparov modules. 
Suppose that $(\A,{}_{\pi_1\hot1}(E_1\hot_BE_2)_C,\D)$ is an unbounded Kasparov module such that:
\begin{enumerate}
\item for all $\psi$ in a dense subspace of $\A\cdot\Dom\D_1$, we have \(\til T_\psi \in \Lip(\D\oplus\D_2)\). 
\item we have the domain inclusion $\Dom\D\subset\Dom\D_1\hot1$;
\item there exists $c\in\R$ such that for all $\psi\in\Dom(\D)$ we have 
\[
\la(\D_1\hot1)\psi \mvert \D \psi\ra + \la\D \psi \mvert (\D_1\hot1)\psi\ra \geq c \la\psi|\psi\ra . 
\]
\end{enumerate}
Then $(\A,{}_{\pi_1\hot1}(E_1\hot_BE_2)_C,\D)$ represents the Kasparov product of $(\A, {}_{\pi_1}(E_1)_B,\D_1)$ and $(\B,{}_{\pi_2}(E_2)_C,\D_2)$. 
\end{thm}

In \cref{sec:Kasp_prod_revisited} we will show that the lower bound in the positivity condition (3) can be weakened to a \emph{relative} bound. Moreover, using the localisation procedure from \cref{sec:localisations}, we will provide several variants of the above theorem for half-closed modules in \cref{sec:Kasp_prod}.

\subsection{Regular symmetric operators}
\label{sec:reg_symm}

Let $B$ be a $\Z_2$-graded $\sigma$-unital $C^*$-algebra, and let $E$ be a $\Z_2$-graded countably generated Hilbert $B$-module.  
Throughout this subsection, we consider a regular symmetric operator $\D$ on $E$ and a positive number $r\in(0,\infty)$. 
For $\lambda\in[0,\infty)$, we introduce the notation 
\[
R_\D^r(\lambda) := (r^2+\lambda+\D^*\D)^{-1} . 
\]
We recall that we have the integral formula
\begin{align}
\label{eq:integral_formula}
(r^2+\D^*\D)^{-\frac12} = \frac1\pi \int_0^\infty \lambda^{-1/2} R_\D^r(\lambda) d\lambda ,
\end{align}
where the integral converges in norm. 
Let us consider the `bounded transform'
\[
F_\D^r := \D (r^2+\D^*\D)^{-\frac12} . 
\]
The following lemma is exactly as in \cite[Lemma 2.7]{vdD20_Kasp_open}, except that we have replaced $1+\D^*\D$ by $r^2+\D^*\D$. 

\begin{lem}
\label{lem:integral_formula}
For all $\psi\in E$ we have
\[
\frac1\pi \int_0^\infty \lambda^{-1/2} \D R_\D^r(\lambda) \psi d\lambda = F_\D^r \psi .
\]
Moreover, for any continuous function $f\colon\R\to\R$ such that $f(x^2)(1+x^2)^{-\frac12}$ is bounded, we also have 
\[
\frac1\pi \int_0^\infty \lambda^{-1/2} f(\D^*\D) R_\D^r(\lambda) \psi d\lambda = f(\D^*\D)(r^2+\D^*\D)^{-1/2} \psi .
\]
\end{lem}

The following lemma was proven on Hilbert spaces in \cite[Proposition A.1]{Les05}, and it was shown in \cite[Lemma 7.7]{LM19} that the argument can be generalised to Hilbert modules. 

\begin{lem}[{cf.\ \cite[Lemma 7.7]{LM19}}]
\label{lem:interpolation}
Let $P$ be an invertible regular positive self-adjoint operator on $E$, and let $T$ be a symmetric operator on $E$ with $\Dom P \subset \Dom T$. 
If $TP^{-1}$ is bounded, then the densely defined operator $P^{-\frac12}TP^{-\frac12}$ extends to an adjointable endomorphism on $E$, and $\|P^{-\frac12}TP^{-\frac12}\| \leq \|TP^{-1}\|$. 
\end{lem}

\begin{lem}
\label{lem:comm_cpt}
Let $a=a^*\in\Lip(\D)$, and consider $b\in\End_B(E)$ such that $(1+\D^*\D)^{-\frac12}b$ is compact on $E$. 
Then the following statements hold:
\begin{enumerate}
\item \label{lem:comm_cpt_R}
The operators $[\D R_\D^r(\lambda),a] b$ and $(1+\lambda)^{\frac12} [R_\D^r(\lambda),a] b$ are compact and of order $\mO(\lambda^{-1})$. 
\item \label{lem:comm_cpt_F}
The operator $[F_\D^r,a] b$ is compact. 
\end{enumerate}
\end{lem}
\begin{proof}
The proof of \ref{lem:comm_cpt_R} is an abstract generalisation of \cite[Lemma 2.9]{vdD20_Kasp_open}. 
We have
\begin{align*}
[R_\D^r(\lambda),a] b &= \big[ (r^2+\lambda+\D^*\D)^{-1} , a \big]  b = - (r^2+\lambda+\D^*\D)^{-1} [\D^*\D,a]  (r^2+\lambda+\D^*\D)^{-1} b .
\end{align*}
We note that, a priori, $[\D^*\D,a] $ may not be well-defined, since it is not clear if $a$ preserves $\Dom\D^*\D$. However, rewriting $[\D^*\D,a]  = \D^* [\D,a]  + (-1)^{\deg a} [\D^*,a]  \D$, we obtain the well-defined expression 
\begin{multline*}
[R_\D^r(\lambda),a]  b 
= 
- R_\D^r(\lambda)^{\frac12} \Big( R_\D^r(\lambda)^{\frac12} [\D^*,a]  \big( \D R_\D^r(\lambda)^{\frac12} \big) + (-1)^{\deg a} \big( R_\D^r(\lambda)^{\frac12} \D^* \big) [\D,a]  R_\D^r(\lambda)^{\frac12} \Big) R_\D^r(\lambda)^{\frac12} b .
\end{multline*}
Since $R_\D^r(\lambda)^{\frac12} b$ is compact, one sees that the right-hand-side of this expression is compact and of order $\mO(\lambda^{-\frac32})$. Moreover, $\D [R_\D^r(\lambda),a]  b$ is also well-defined, compact, and of order $\mO(\lambda^{-1})$. Finally, we see that 
\[
[\D R_\D^r(\lambda),a]  b = [\D,a]  R_\D^r(\lambda) b + \D [R_\D^r(\lambda),a]  b 
\]
is compact and of order $\mO(\lambda^{-1})$. Thus we have proven \ref{lem:comm_cpt_R}. 

Using \cref{lem:integral_formula}, we have for any $\psi\in E$ that 
\begin{align*}
\big[ F_\D^r , a \big]  b \psi 
&= - \frac1\pi \int_0^\infty \lambda^{-1/2} [\D R_\D^r(\lambda),a]  b \psi d\lambda .
\end{align*}
By \ref{lem:comm_cpt_R}, $[\D R_\D^r(\lambda),a]  b$ is compact and of order $\mO(\lambda^{-1})$. Hence the above integral converges in norm to a compact operator, which proves \ref{lem:comm_cpt_F}. 
\qedhere
\end{proof}

\begin{lem}
\label{lem:adjoint_cpt_res}
Let $a\in\End_B(E)$ be such that $a(1+\D^*\D)^{-\frac12}$ is compact, and let $b\in\Lip^*(\D)$.
Then $ab(1+\D\D^*)^{-\frac12}$ is compact. 
\end{lem}
\begin{proof}
Consider the domain inclusions $\iota\colon\Dom\D\into E$ and $\bar\iota\colon\Dom\D^*\into E$. By assumption, $a\circ\iota$ is compact. Furthermore, $b$ maps $\Dom\D^*$ into $\Dom\D$, and $b\colon\Dom\D^*\to\Dom\D$ is bounded with respect to the graph norms (indeed, its norm is bounded by $\|b\|+\|[\D,b] \|$). Moreover, by \cite[Lemma 2.2 \& Remark 2.4]{Hil10}, $b\colon\Dom\D^*\to\Dom\D$ is adjointable. 
Hence $ab\circ\bar\iota = (a\circ\iota) \circ b\colon \Dom\D^*\to E$ is the composition of an adjointable and a compact operator, and therefore it is compact. 
Since $(1+\D\D^*)^{-\frac12}$ is a bounded map from $E$ to $\Dom\D^*$, the statement follows. 
\end{proof}

\subsection{The connection condition}
\label{sec:connection}

\begin{defn}
\label{defn:connection_condition}
Consider three half-closed modules $(\A,{}_{\pi_1}(E_1)_B,\D_1)$, $(\B,{}_{\pi_2}(E_2)_C,\D_2)$, and $(\A,{}_\pi E_C,\D)$, where $E := E_1\hot_B E_2$ and $\pi = \pi_1\hot1$, and suppose that $\pi_1$ is essential. 
The \emph{connection condition} requires that for all $\psi$ in a dense subspace $\E_1$ of $\Dom\D_1$, we have 
\[
\til T_\psi := \mattwo{0}{T_\psi}{T_\psi^*}{0} \in \Lip(\D\oplus\D_2) . 
\]
\end{defn}

We can adapt the argument from \cite[Proposition 14]{Kuc97} to symmetric operators (see also \cite[Proposition 6.11]{KS19}), to obtain the following result. 
For simplicity, we restrict our attention to the case where $\pi_1$ is essential (since this is our case of interest in later sections). However, we note that the case of non-essential representations $\pi_1$ can be dealt with in the same way as in \cite[Proposition 14]{Kuc97}. 

\begin{prop}
\label{prop:connection}
The connection condition of \cref{defn:connection_condition} implies the connection condition of \cref{thm:Connes-Skandalis} for $F=F_\D$ and $F_2=F_{\D_2}$. 
\end{prop}
\begin{proof}
It suffices to consider elements $a \psi b \in E_1$, for even elements $a\in\A$ and $b\in\B$, and homogeneous $\psi\in\E_1\subset\Dom\D_1$, where $\E_1$ is the dense subset from the connection condition. 
Since $T_{a\psi b} = a T_{\psi} b$ and $T_{a\psi b}^* = b^* T_{\psi}^* a^*$, we have 
\[
\til T_{a\psi b} = \mattwo{0}{a T_{\psi} b}{b^* T_{\psi}^* a^*}{0} 
= \mattwo{a}{0}{0}{b^*} \mattwo{0}{T_{\psi}}{T_{\psi}^*}{0} \mattwo{a^*}{0}{0}{b} 
=: \til a \til T_{\psi} \til a^* .
\]
Hence the graded commutator with $F_{\D\oplus\D_2}$ is given by 
\begin{align*}
[ F_{\D\oplus\D_2} , \til T_{a\psi b} ] 
&= [ F_{\D\oplus\D_2} , \til a \til T_{\psi} \til a^* ]  \\ 
&= [ F_{\D\oplus\D_2} , \til a ]  \til T_{\psi} \til a^* + \til a [ F_{\D\oplus\D_2} , \til T_{\psi} ]  \til a^* + (-1)^{\deg\psi} \til a \til T_{\psi} [ F_{\D\oplus\D_2} , \til a^* ]  .
\end{align*}
By \cref{defn:connection_condition}, we know that $\til T_{\psi}\in\Lip(\D\oplus\D_2)$, so it follows from \cref{lem:comm_cpt}.\ref{lem:comm_cpt_F} that the second term is compact. Since the first and third terms are also compact (by \cref{thm:Hilsum}), this proves the statement. 
\end{proof}

\section{Kucerovsky's theorem revisited}
\label{sec:Kasp_prod_revisited}

In this section, we consider only the special case of unbounded Kasparov modules (i.e., the case of \emph{self-adjoint} operators). 
Thus we consider the following setting. 

\begin{assumption}
\label{ass:Kasp_prod_revisited}
Let $A$ be a ($\Z_2$-graded) separable $C^*$-algebra, let $B$ and $C$ be ($\Z_2$-graded) $\sigma$-unital $C^*$-algebras, and let $\A\subset A$ and $\B\subset B$ be dense $*$-subalgebras. 
Consider three unbounded Kasparov modules $(\A,{}_{\pi_1}(E_1)_B,\D_1)$, $(\B,{}_{\pi_2}(E_2)_C,\D_2)$, and $(\A,{}_\pi E_C,\D)$, where $E := E_1\hot_B E_2$ and $\pi = \pi_1\hot1$. 
\end{assumption}

Our aim in this section is to improve Kucerovsky's \cref{thm:Kucerovsky} by weakening the positivity condition, as follows. 

\begin{defn}
\label{defn:positivity_condition}
In the setting of \cref{ass:Kasp_prod_revisited}, 
the \emph{positivity condition} requires that the following assumptions hold: 
\begin{enumerate}
\item we have the domain inclusion $\Dom(\D)\subset\Dom(\D_1\hot1)$; 
\item there exists $c\in[0,\infty)$ such that for all $\psi\in\Dom(\D)$ we have 
\[
\big\la (\D_1\hot1)\psi \bigmvert \D\psi \big\ra + \big\la \D\psi \bigmvert (\D_1\hot1)\psi \big\ra \geq - c \big\la \psi \bigmvert (1+\D^2)^{\frac12}\psi \big\ra . 
\] 
\end{enumerate}
\end{defn}

We then obtain the following improvement of \cref{thm:Kucerovsky}. 

\begin{thm}
\label{thm:Kucerovsky_revisited}
In the setting of \cref{ass:Kasp_prod_revisited}, assume furthermore that the connection condition (\cref{defn:connection_condition}) and the positivity condition (\cref{defn:positivity_condition}) are satisfied. 
Then $(\A,{}_\pi E_C,\D)$ represents the Kasparov product of $(\A,{}_{\pi_1}(E_1)_B,\D_1)$ and $(\B,{}_{\pi_2}(E_2)_C,\D_2)$. 
\end{thm}

Before proving the theorem, we first describe a few consequences. 
First of all, in the \emph{bounded} picture, we know that two Kasparov modules $(A,{}_\pi E_B,F)$ and $(A,{}_\pi E_B,F')$ are equivalent if for each $a\in A$, the operator $a^*[F,F']a$ is positive modulo compact operators (see \cite[Lemma 11]{Ska84} or \cite[Proposition 17.2.7]{Blackadar98}). 
The following statement gives an unbounded analogue (albeit a `global' analogue, since we no longer `localise' by elements $a\in A$), which says, roughly speaking, that two unbounded Kasparov modules $(\A,{}_\pi E_B,\D)$ and $(\A,{}_\pi E_B,\D')$ are equivalent if $[\D,\D']$ is positive modulo `first-order operators' (this generalises \cite[Corollary 17]{Kuc97}). 

\begin{coro}
\label{coro:equivalent}
Let $(\A,{}_\pi E_B,\D)$ and $(\A,{}_\pi E_B,\D')$ be unbounded Kasparov modules such that $\Dom\D\subset\Dom\D'$, and suppose there exists $c\in[0,\infty)$ such that for all $\psi\in\Dom(\D)$ we have 
\[
\la \D'\psi \mvert \D\psi\ra + \la \D\psi \mvert \D'\psi\ra \geq - c \big\la \psi \bigmvert (1+\D^2)^{\frac12}\psi \big\ra . 
\]
Then the two unbounded Kasparov modules $(\A,{}_\pi E_B,\D)$ and $(\A,{}_\pi E_B,\D')$ are homotopy-equivalent, i.e.\ $[\D']=[\D]\in\KK(A,B)$. 
\end{coro}
\begin{proof}
We claim that $[\D]$ equals the internal Kasparov product (over $B$) of $[\D']$ with $1_B\in\KK(B,B)$. The class $1_B$ is represented by the unbounded Kasparov module $(B,B_B,0)$. We can identify $E\hot_BB\simeq E$, and then for each $\psi\in E$, the map $T_\psi\colon B\to E$ is given by $b\mapsto \psi b$. Its adjoint $T_\psi^*\colon E\to B$ is given by $\phi\mapsto\la\psi|\phi\ra$. For each $\psi\in\Dom\D$, we then see that the operators $\D T_\psi = T_{\D\psi}$ and $T_\psi^*\D = T_{\D\psi}^*$ are both bounded. Hence the connection condition is satisfied. 
Since the positivity condition holds by hypothesis, the claim follows from \cref{thm:Kucerovsky_revisited}. 
\end{proof}

The following lemma provides a sufficient condition for the positivity condition, given in terms of operators instead of form estimates. This sufficient condition is in particular useful for the construction of the unbounded Kasparov product from weakly anti-commuting operators (see \cref{coro:LM}). 

\begin{lem}
\label{lem:positivity_sufficient}
Let $\D$ and $\mS$ be odd regular self-adjoint operators on a $\Z_2$-graded Hilbert $C$-module $E$, such that $\Dom\D \subset \Dom \mS$. 
Suppose there exists a core $\mF\subset\Dom\D$ such that $\mS\cdot\mF\subset\Dom\D$ and $\D\cdot\mF\subset\Dom \mS$, so that $[\D,\mS]$ is well-defined on $\mF$. 
Assume that on $\mF$ we have the equality $[\D,\mS] = P + R$, where $P$ is a densely defined positive symmetric operator on $\mF$ (i.e.\ $\mF\subset\Dom P$ and $\la\psi|P\psi\ra\geq0$ for all $\psi\in\mF$), and $R$ is a densely defined symmetric operator which is relatively bounded by $\D$ (i.e.\ $\Dom\D\subset\Dom R$). 
Then for all $\psi\in\Dom\D$ we have 
\begin{align*}
\la \mS\psi \mvert \D\psi\ra + \la \D\psi \mvert \mS\psi\ra \geq - \| R (1+\D^2)^{-\frac12} \| \, \big\la \psi \bigmvert (1+\D^2)^{\frac12}\psi \big\ra . 
\end{align*}
\end{lem}
\begin{proof}
From \cref{lem:interpolation} we know that $\big\| (1+\D^2)^{-\frac14} R (1+\D^2)^{-\frac14} \big\| \leq c := \| R (1+\D^2)^{-\frac12} \|$, and hence we have for $\psi\in\Dom\D$ the inequality 
\[
\pm \la\psi\mvert R\psi\ra = \pm \big\la (1+\D^2)^{\frac14} \psi \bigmvert (1+\D^2)^{-\frac14} R (1+\D^2)^{-\frac14} \, (1+\D^2)^{\frac14} \psi \big\ra \leq c \big\la \psi \bigmvert (1+\D^2)^{\frac12} \psi \big\ra . 
\]
Using also the positivity of $P$, we then find for all $\psi\in\mF$ that 
\[
\la \mS\psi \mvert \D\psi\ra + \la \D\psi \mvert \mS\psi\ra = \la\psi\mvert P\psi\ra + \la\psi\mvert R\psi\ra \geq - c \big\la \psi \bigmvert (1+\D^2)^{\frac12}\psi \big\ra . 
\]
For arbitrary $\psi\in\Dom\D$, we choose a sequence $\psi_n\in\mF$ such that $\|\psi_n-\psi\|_\D\to0$ as $n\to\infty$. 
Since $\Dom\D\subset\Dom \mS$, we note that we then also have the convergence $\|\mS\psi_n-\mS\psi\|\to0$. 
Applying the above inequality to $\psi_n$ we obtain
\begin{align*}
\la \mS\psi \mvert \D\psi\ra + \la \D\psi \mvert \mS\psi\ra 
&= \lim_{n\to\infty} \la \mS\psi_n \mvert \D\psi_n\ra + \la \D\psi_n \mvert \mS\psi_n\ra \\
&\geq - c \lim_{n\to\infty} \big\la \psi_n \bigmvert (1+\D^2)^{\frac12}\psi_n \big\ra \\
&= - c \big\la \psi \bigmvert (1+\D^2)^{\frac12}\psi \big\ra . 
\qedhere
\end{align*}
\end{proof}

Let $\mS$ and $\mT$ be odd regular self-adjoint operators on a $\Z_2$-graded Hilbert $C$-module $E$. 
We consider the linear subspace 
\[
\mF(\mS,\mT) := \Dom(\mS\mT)\cap\Dom(\mT\mS) 
= \{\psi\in\Dom \mS\cap\Dom \mT : \mS\psi\in\Dom \mT , \; \mT\psi\in\Dom \mS\} .
\]
Then $\mS$ and $\mT$ are called \emph{weakly anti-commuting} \cite[Definition 2.1]{LM19} if 
\begin{itemize}
\item there is a constant $C>0$ such that for all $\psi\in\mF(\mS,\mT)$ we have 
\[
\big\la [\mS,\mT] \psi \bigmvert [\mS,\mT] \psi \big\ra \leq C \big( \la \psi|\psi\ra + \la \mS\psi|\mS\psi\ra + \la \mT\psi|\mT\psi\ra \big) ;
\]
\item there is a core $\E\subset\Dom \mT$ such that $(\mS+\lambda)^{-1}(\E) \subset \mF(\mS,\mT)$ for $\lambda\in i\R$, $|\lambda|\geq\lambda_0>0$. 
\end{itemize}

We now obtain a different proof of the following result due to Lesch and Mesland. 

\begin{coro}[{\cite[Theorem 7.4]{LM19}}]
\label{coro:LM}
Suppose we are given two unbounded Kasparov modules $(\A,{}_{\pi_1}(E_1)_B,\D_1)$ and $(\B,{}_{\pi_2}(E_2)_C,\D_2)$. 
We write $E := E_1\hot_B E_2$, $\pi = \pi_1\hot1$, and $\mS := \D_1\hot1$. 
Let $\mT$ be an odd regular self-adjoint operator on $E$, and consider the operator $\D := \mS+\mT$ on the domain $\Dom\D := \Dom \mS\cap\Dom \mT$. 
We assume that the following conditions hold:
\begin{enumerate}
\item \label{item:connection_T}
for all $\psi$ in a dense subset of $\Dom\D_1$, we have
\[
\til T_\psi := \mattwo{0}{T_\psi}{T_\psi^*}{0} \in \Lip(\mT\oplus\D_2) ; 
\]
\item \label{item:Lip_T}
we have $\A\subset\Lip(\mT)$; 
\item \label{item:anti-comm_S-T}
$\mS$ and $\mT$ are weakly anti-commuting. 
\end{enumerate}
Then $(\A,{}_{\pi}(E)_C,\D)$ is an unbounded Kasparov module that represents the Kasparov product of $(\A,{}_{\pi_1}(E_1)_B,\D_1)$ and $(\B,{}_{\pi_2}(E_2)_C,\D_2)$. 
\end{coro}
\begin{proof}
By \ref{item:anti-comm_S-T} and \cite[Theorem 2.6]{LM19}, $\D$ is regular and self-adjoint. Using \ref{item:Lip_T} we clearly have $\A \subset \Lip(\mT)\cap\Lip(\mS) \subset \Lip(\D)$. As in the proof of \cite[Theorem 7.4]{LM19}, we know that $a(\D\pm i)^{-1}$ is compact for any $a\in A$. 
Thus $(\A,{}_{\pi}(E)_C,\D)$ is indeed an unbounded Kasparov module. For any $\psi\in\Dom\D_1$ we have bounded operators $\mS T_\psi = T_{\D_1\psi}$ and $\mT^*_\psi \mS = \mT^*_{\D_1\psi}$, which means that $\til T_\psi\in\Lip(\mS\oplus0)$. Combined with \ref{item:connection_T} this ensures that the connection condition (\cref{defn:connection_condition}) is satisfied. 

We note that $\mF := \mF(\mS,\mT)$ is a core for $\D$ \cite[Theorem 2.6.(4)]{LM19}, that $\mS\cdot\mF\subset\Dom \mS$ \cite[Theorem 5.1]{LM19}, and that $[\mS,\mT]$ is relatively bounded by $\D$ \cite[Theorem 2.6.(1)]{LM19}. 
Thus \cref{lem:positivity_sufficient} applies with $\mF=\mF(\mS,\mT)$, $P=\mS^2$ and $R=[\mS,\mT]$, and we see that the positivity condition (\cref{defn:positivity_condition}) is also satisfied. 
Hence it follows from \cref{thm:Kucerovsky_revisited} that $(\A,{}_{\pi}(E)_C,\D)$ represents the Kasparov product of $(\A,{}_{\pi_1}(E_1)_B,\D_1)$ and $(\B,{}_{\pi_2}(E_2)_C,\D_2)$. 
\end{proof}

\begin{remark}
\label{remark:construct_T}
Under suitable assumptions \cite{Mes14,KL13,BMS16}, one can \emph{construct} an operator $\mT$ of the form
\[
\mT = 1\hot_\nabla\D_2 ,
\]
where $\nabla$ is a suitable `connection' on $E_1$, and prove that $\mT$ satisfies the conditions of \cref{coro:LM}. 
In many geometric examples, the connection $\nabla$ is naturally determined by the given geometry \cite{BMS16,KS18,KS20,vdD20_Kasp_open}. 
In fact, such an operator $\mT$ always exists, if one is willing to allow $(\A,{}_{\pi_1}(E_1)_B,\D_1)$ and $(\B,{}_{\pi_2}(E_2)_C,\D_2)$ to be replaced by homotopy-equivalent modules \cite{MR16}. 
\end{remark}

\subsection{Proof of \texorpdfstring{\cref{thm:Kucerovsky_revisited}}{the Theorem}}

Before we proceed with the proof, let us first introduce some notation. 
\begin{notation}
\label{notation}
Let $\D$ and $\mS$ be regular self-adjoint operators on a Hilbert $C$-module $E$, such that $\Dom\D\subset\Dom \mS$. 
We consider the quadratic form $Q$ defined for $\psi \in \Dom\D$ by
\begin{align*}
Q(\psi) &:= 2 \Re \la \D \psi | \mS \psi \ra = \la \D \psi | \mS \psi \ra + \la \mS \psi | \D \psi \ra . 
\end{align*}
For $\lambda,\mu\in[0,\infty)$, we use the notation 
\begin{align*}
R_\D(\lambda) &:= (1+\lambda+\D^2)^{-1} , & 
R_\mS(\mu) &:= (1+\mu+\mS^2)^{-1} . 
\end{align*}
We introduce the following bounded operators: 
\begin{align*}
k_\D(\lambda) &:= \sqrt{1+\lambda} R_\D(\lambda), & 
k_\mS(\mu) &:= \sqrt{1+\mu} R_\mS(\mu) , \\
h_\D(\lambda) &:= \D R_\D(\lambda) , & 
h_\mS(\mu) &:= \mS R_\mS(\mu) . 
\end{align*}
Furthermore, we define 
\begin{align*}
M_1(\lambda,\mu) &:= h_\D(\lambda) h_\mS(\mu) , & 
M_2(\lambda,\mu) &:= k_\D(\lambda) h_\mS(\mu) , \\
M_3(\lambda,\mu) &:= h_\D(\lambda) k_\mS(\mu) , & 
M_4(\lambda,\mu) &:= k_\D(\lambda) k_\mS(\mu) .
\end{align*}
\end{notation}

\begin{lem}
\label{lem:integral_estimate_first-order}
For $\psi\in E$, we have the inequality 
\begin{align*}
\sum_{m=1}^4 \frac{1}{\pi^2} \int_0^\infty \int_0^\infty (\mu\lambda)^{-\frac12} \big\la M_m(\lambda,\mu)\psi \bigmvert (1+\D^2)^{\frac12} M_m(\lambda,\mu)\psi \big\ra d\lambda d\mu 
\leq \la\psi|\psi\ra . 
\end{align*}
\end{lem}
\begin{proof}
Let us consider the four integrals (for $m=1,2,3,4$) given by 
\[
I_m := \frac{1}{\pi^2} \int_0^\infty \int_0^\infty (\mu\lambda)^{-\frac12} \big\la M_m(\lambda,\mu) \psi \bigmvert (1+\D^2)^{\frac12} M_m(\lambda,\mu) \psi \big\ra d\lambda d\mu .
\]
We note that $h_\D(\lambda)^2 + k_\D(\lambda)^2 = R_\D(\lambda)$. 
Moreover, by \cref{lem:integral_formula}
we have the strongly convergent integral 
\[
\frac{1}{\pi} \int_0^\infty \lambda^{-\frac12} (1+\D^2)^{\frac12} \big( h_\D(\lambda)^2 + k_\D(\lambda)^2 \big) d\lambda 
= \frac{1}{\pi} \int_0^\infty \lambda^{-\frac12} (1+\D^2)^{\frac12} R_\D(\lambda) d\lambda 
= 1 .
\]
Computing the integrals over $\lambda$, we thus obtain the equalities 
\begin{align*}
I_1 + I_2 &= \frac{1}{\pi} \int_0^\infty \mu^{-\frac12} \big\la h_\mS(\mu) \psi \bigmvert h_\mS(\mu) \psi \big\ra d\mu , \\
I_3 + I_4 &= \frac{1}{\pi} \int_0^\infty \mu^{-\frac12} \big\la k_\mS(\mu) \psi \bigmvert k_\mS(\mu) \psi \big\ra d\mu . 
\end{align*}
Summing up the latter two equalities and computing the remaining norm-convergent integral over $\mu$, we obtain
\begin{align*}
\sum_{m=1}^4 I_m &= \frac{1}{\pi}\int_0^\infty \mu^{-\frac12} \big\la \psi \bigmvert R_\mS(\mu) \psi \big\ra d\mu 
= \big\la \psi \bigmvert (1+\mS^2)^{-\frac12} \psi \big\ra 
\leq \big\la \psi \bigmvert \psi \big\ra . 
\qedhere
\end{align*}
\end{proof}

The following result relates our positivity condition to the positivity condition of Connes-Skandalis (\cref{thm:Connes-Skandalis}). 

\begin{prop}
\label{prop:revisited_positivity}
Let $\D$ and $\mS$ be odd regular self-adjoint operators on a $\Z_2$-graded Hilbert $C$-module $E$, such that $\Dom\D \subset \Dom \mS$. 
Suppose there exists a constant $c\in[0,\infty)$ such that for all $\psi\in\Dom\D$ we have 
\begin{align}
\label{eq:unbdd_pos_revisited}
\la \mS\psi \mvert \D\psi\ra + \la \D\psi \mvert \mS\psi\ra \geq - c \big\la \psi \bigmvert (1+\D^2)^{\frac12}\psi \big\ra . 
\end{align}
Then for any $0<\kappa<2$ there exists an $\alpha\in(0,\infty)$ such that $[ F_\D , F_{\alpha \mS} ] + \kappa$ is positive: 
\[
[F_\D , F_{\alpha \mS}] \geq - \kappa .
\]
\end{prop}
\begin{proof}
Recall that for $\psi\in E$ we have by \cref{lem:integral_formula} that $F_\D \psi = \frac1\pi \int_0^\infty \lambda^{-\frac12} \D R_\D(\lambda) \psi d\lambda$, and similarly for $F_\mS$. 
Applying \cref{lem:integral_formula} twice, we can then rewrite 
\begin{align*}
\big\la \psi \bigmvert [F_\D , F_\mS] \psi \big\ra 
&= 2 \Re \big\la \psi \bigmvert F_\D F_\mS \psi \big\ra \nonumber\\
&= \frac{1}{\pi^2} \int_0^\infty \int_0^\infty (\mu\lambda)^{-\frac12} 2 \Re \big\la \psi \bigmvert \D R_\D(\lambda) \mS R_\mS(\mu) \psi \big\ra d\lambda d\mu .
\end{align*}
By the same computation as in \cite[Lemma 11]{Kuc97} (or as in \cref{lem:positivity_local_expr} below, taking the special case $v=\phi=\rho=1$), 
the integrand on the right-hand-side can be rewritten as
\begin{align*}
2 \Re \big\la \psi \bigmvert \D R_\D(\lambda) \mS R_\mS(\mu) \psi \big\ra 
= \sum_{m=1}^4 Q\big(M_m(\lambda,\mu)\psi\big) . 
\end{align*}
By \cref{eq:unbdd_pos_revisited} we have $Q(\psi) \geq - c \la \psi \mvert (1+\D^2)^{\frac12}\psi\ra$, and therefore we obtain 
\begin{align*}
\big\la \psi \bigmvert [F_\D , F_\mS] \psi \big\ra 
&\geq - c \sum_{m=1}^4 \frac{1}{\pi^2} \int_0^\infty \int_0^\infty (\mu\lambda)^{-\frac12} \big\la M_m(\lambda,\mu)\psi \bigmvert (1+\D^2)^{\frac12}M_m(\lambda,\mu)\psi \big\ra d\lambda d\mu .
\end{align*}
Applying the inequality from \cref{lem:integral_estimate_first-order}, we conclude that 
\[
\big\la \psi \bigmvert [F_\D , F_\mS] \psi \big\ra \geq - c \la\psi|\psi\ra . 
\]
Finally, if we replace $\mS$ by $\alpha \mS$ for some $\alpha>0$, then we see from \cref{eq:unbdd_pos_revisited} that $c$ should be replaced by $\alpha c$. 
Thus, by choosing $\alpha$ small enough, we can ensure that $\alpha c  < \kappa < 2$. 
\end{proof}

\begin{proof}[\textbf{Proof of \cref{thm:Kucerovsky_revisited}}]
We can represent $[\D]$, $[\D_1]=[\alpha\D_1]$ (for some $\alpha\in(0,\infty)$), and $[\D_2]$ by their bounded transforms $F_\D$, $F_{\alpha\D_1}$, and $F_{\D_2}$, respectively. 
The statement of the theorem follows from \cref{thm:Connes-Skandalis}, where the connection condition is satisfied by \cite[Proposition 14]{Kuc97} (see also \cref{prop:connection} if $\pi_1$ is essential), and (choosing $\alpha$ small enough) the positivity condition is satisfied by \cref{prop:revisited_positivity}. 
\end{proof}

\section{Localisations of half-closed modules}
\label{sec:localisations}

\subsection{Localisations of unbounded operators}
\label{sec:localisations_operators}

Let $\D$ be a regular symmetric operator on a Hilbert $B$-module $E$. For any $b=b^*\in\Lip^*(\D)$, we will consider the \emph{localisation} of $\D$ given by the operator $b\D b$. We recall the following lemma. 

\begin{lem}[{\cite[Lemma 3.2]{KS19}}]
\label{lem:loc_sa}
Let $\D$ be a regular symmetric operator on a Hilbert $B$-module $E$, and let $b=b^*\in\Lip^*(\D)$. 
Then (the closure of) $b\D b$ is regular and self-adjoint, and $\Dom\D$ is a core for $b\D b$. 
\end{lem}

\begin{lem}
\label{lem:cpt_res_localisation}
Let $(\A,E_B,\D)$ be a half-closed module, and consider homogeneous self-adjoint elements $a,b\in\A$ such that $ab=a$. 
Then $a(b\D b\pm i)^{-1}$ is a compact endomorphism. 
\end{lem}
\begin{proof}
Consider the regular self-adjoint operator $\til\D := \mattwo{0}{\D^*}{\D}{0}$, and write $a = \mattwo{a}{0}{0}{a}$ and $b = \mattwo{b}{0}{0}{b}$. 
Repeatedly using $ab=a$, we see that 
\[
(\til\D\pm i) a 
= (\til\D\pm i) ab 
= [\til\D,a]  b + a\big((-1)^{\deg a} \til\D\pm i\big)b 
= [\til\D,a]  b + a\big((-1)^{\deg a} b\til\D b\pm i\big) .
\]
We multiply from the right by $(b\D b\pm i)^{-1}$ and from the left by 
\[
(\til\D\pm i)^{-1} = \mattwo{\pm i}{\D^*}{\D}{\pm i}^{-1} = \mattwo{\mp i(1+\D^*\D)^{-1}}{(1+\D^*\D)^{-1}\D^*}{(1+\D\D^*)^{-1}\D}{\mp i(1+\D\D^*)^{-1}} .
\]
Using that $[\D^*,a]  = [\D,a] $ 
(more precisely, their closures are equal) 
and that $b\D^*b = b\D b$ on $\Dom b\D b$, this yields
\begin{align*}
&\mattwo{a (b\D b\pm i)^{-1}}{0}{0}{a (b\D b\pm i)^{-1}} 
= b^2 \mattwo{a (b\D b\pm i)^{-1}}{0}{0}{a (b\D b\pm i)^{-1}} \\
&\qquad= b^2 \mattwo{\mp i(1+\D^*\D)^{-1}}{(1+\D^*\D)^{-1}\D^*}{(1+\D\D^*)^{-1}\D}{\mp i(1+\D\D^*)^{-1}} \times \\
&\qquad\qquad\times \left( [\D,a]  b \mattwo{0}{1}{1}{0} + (-1)^{\deg a} a \mattwo{\pm i}{(-1)^{\deg a} b\D b}{(-1)^{\deg a} b\D b}{\pm i} \right) (b\D b\pm i)^{-1} .
\end{align*}
We note that $b^2 (1+\D^*\D)^{-\frac12}$ is compact, and by \cref{lem:adjoint_cpt_res} also $b^2 (1+\D\D^*)^{-\frac12}$ is compact. 
Since $\mattwo{\pm i}{(-1)^{\deg a} b\D b}{(-1)^{\deg a} b\D b}{\pm i} (b\D b\pm i)^{-1}$ is bounded, it follows that $a (b\D b\pm i)^{-1}$ is compact. 
\end{proof}

\begin{lem}
\label{lem:local_comparison}
Let $(\A,E_B,\D)$ be a half-closed module. 
Consider homogeneous self-adjoint elements $a\in A$ and $c,b\in\A$ such that $b,c$ are even, $ac=a$, and $cb=c$. Let $\D_b := b\D b$. 
Then $a(F_\D-F_{\D_b})$ is compact on $E$. 
\end{lem}
\begin{proof}
The proof closely follows the argument of \cite[Lemma 2.10]{vdD20_Kasp_open} (which in turn was inspired by \cite[Lemma 3.1]{Hil10}). The main difference here is that we have to take care of the fact that the operators $c(\D_b-\D)$ and $c(\D_b-\D^*)$ do not vanish (see below). 

Since $a(F_\D-F_\D^*)$ is compact by \cref{thm:Hilsum}, it suffices to show that $a(F_\D^*-F_{\D_b})$ is compact. 
We can rewrite 
\begin{align*}
a(F_\D^*-F_{\D_b}) 
&= a \left( \bar{(1+\D^*\D)^{-\frac12} \D^*} - \bar{\D_b (1+\D_b^2)^{-\frac12}} \right) \\
&= a \left( \bar{\D^* (1+\D\D^*)^{-\frac12}} - \bar{\D_b (1+\D_b^2)^{-\frac12}} \right) . 
\end{align*}
Using \cref{lem:integral_formula}, we have for any $\psi\in E$ that 
\[
a(F_\D^*-F_{\D_b}) \psi = \frac1\pi \int_0^\infty \lambda^{-\frac12} T(\lambda) \psi d\lambda ,
\]
where 
\[
T(\lambda) := a \left( \D^* (1+\lambda+\D\D^*)^{-1} - (1+\lambda+\D_b^2)^{-1} \D_b \right) .
\]
We claim that $T(\lambda)$ is a compact operator on $E$, and that $\|T(\lambda)\|$ is of order $\mO(\lambda^{-1})$ as $\lambda\to\infty$. It then follows that $\frac1\pi \int_0^\infty \lambda^{-\frac12} T(\lambda) d\lambda$ is in fact a \emph{norm}-convergent integral of compact operators, which proves the statement. 
To prove the claim, we rewrite 
\begin{align*}
T(\lambda) 
&= a \bar{(1+\lambda+\D_b^2)^{-1} (1+\lambda+\D_b^2)} \; \bar{\D^* (1+\lambda+\D\D^*)^{-1}} \\
&\quad- a \bar{(1+\lambda+\D_b^2)^{-1} \D_b} \; \bar{(1+\lambda+\D\D^*) (1+\lambda+\D\D^*)^{-1}} \\
&= a \bar{(1+\lambda+\D_b^2)^{-1} \D_b} \; \bar{(\D_b - \D) \D^* (1+\lambda+\D\D^*)^{-1}} \\
&\quad+ (1+\lambda) a (1+\lambda+\D_b^2)^{-1} \bar{(\D^* - \D_b) (1+\lambda+\D\D^*)^{-1}} .
\end{align*}
We note that the operators on the last line are still well-defined. For instance, we have $\Ran\big(\D^* (1+\lambda+\D\D^*)^{-1}\big) \subset \Dom\D \subset \Dom\D_b$, so that $\bar{(\D_b - \D) \D^* (1+\lambda+\D\D^*)^{-1}}$ is a well-defined bounded operator. We also note that $b\cdot\Dom\D^*\subset\Dom\D$, so that $\D_b$ is well-defined on $\Dom\D^*$. 

Next, since $cb = c$, we note that $c(\D_b-\D) = c(b^2-1)\D + cb[\D,b]  = c[\D,b] $ and similarly $c(\D^*-\D_b) = -c[\D^*,b] $. Noting that $\bar{[\D^*,b] } = \bar{[\D,b] }$ 
and using that $ac=a$, we find that 
\begin{align*}
T(\lambda) &= a \big[ c , \bar{(1+\lambda+\D_b^2)^{-1} \D_b} \big]  \bar{(\D_b - \D) \D^* (1+\lambda+\D\D^*)^{-1}} \nonumber\\
&\quad+ a (1+\lambda) \big[ c , (1+\lambda+\D_b^2)^{-1} \big]  \bar{(\D^* - \D_b) (1+\lambda+\D\D^*)^{-1}} \nonumber\\
&\quad+ a \bar{(1+\lambda+\D_b^2)^{-1} \D_b} c\bar{[\D,b] } \bar{\D^* (1+\lambda+\D\D^*)^{-1}} \nonumber\\
&\quad- a (1+\lambda) (1+\lambda+\D_b^2)^{-1} c\bar{[\D,b] } \bar{(1+\lambda+\D\D^*)^{-1}} .
\end{align*}
From \cref{lem:cpt_res_localisation} we know that $c(\D_b\pm i)^{-1}$ and $a(\D_b\pm i)^{-1}$ are compact. In particular, also $a(1+\lambda+\D_b^2)^{-\frac12}$ is compact. 
Furthermore, we may apply \cref{lem:comm_cpt}.\ref{lem:comm_cpt_R} to see that $a \big[ c , \bar{(1+\lambda+\D_b^2)^{-1} \D_b} \big] $ is compact and of order $\mO(\lambda^{-1})$, and that $a \big[ c , (1+\lambda+\D_b^2)^{-1} \big] $ is compact and of order $\mO(\lambda^{-\frac32})$. 
Using these facts, we see that $T(\lambda)$ is indeed compact and of order $\mO(\lambda^{-1})$. 
\end{proof}

\subsection{Construction of the localised representative}
\label{sec:local_representative}

We show here that the construction of a localised representative for vertical operators on submersions of open manifolds, as described in \cite[\S2.4]{vdD20_Kasp_open}, generalises to the abstract (noncommutative) setting of half-closed modules. 
Recall that a (positive, increasing, contractive) approximate unit $\{u_n\}_{n\in\N}$ is called \emph{almost idempotent} if $u_{n+1} u_n = u_n$ for all $n\in\N$ \cite[Definition II.4.1.1]{Blackadar06}. 

\begin{assumption}
\label{ass:half-closed_partition}
We consider a half-closed $A$-$B$-module $(\A,{}_\pi E_B,\D)$ for which the representation $\pi\colon A\to\End_B(E)$ is essential. 
We assume that the $*$-subalgebra $\A\subset A$ contains an (even) almost idempotent approximate unit $\{u_n\}_{n\in\N}$ for $A$. 
\end{assumption}

\begin{remark}
\label{remark:approx_unit}
\begin{enumerate}
\item 
Since $\pi$ is essential, it follows that $\pi(u_n)$ converges strongly to the identity on $E$ (as $n\to\infty$). 
\item 
We know from \cite[Corollary II.4.2.5]{Blackadar06} that a $\sigma$-unital $C^*$-algebra $A$ always contains an almost idempotent approximate unit $\{u_n\}$. 
Our main assumption is that we can find such $\{u_n\}$ \emph{inside the dense $*$-subalgebra $\A$}. 
\item 
In the special case where $A$ is unital, we can of course consider $u_n=1_A$ for all $n\in\N$. 
\end{enumerate}
\end{remark}

\begin{defn}
\label{defn:partition}
The `partition of unity' $\{\chi_k^2\}_{k\in\N}$ corresponding to the approximate unit $\{u_n\}_{n\in\N}$ is defined by 
\begin{align}
\label{eq:approx_unit_to_partition}
\chi_0 &:= u_0^{\frac12} , &
\chi_k &:= (u_k-u_{k-1})^{\frac12} \quad (k>0) .
\end{align}
\end{defn}
While $\chi_k^2 = u_k-u_{k-1}$ always lies in $\A$ for each $k\in\N$, we note that we do not know if also $\chi_k$ lies in $\A$ (we only know $\chi_k\in A$). 

\begin{lem}
\label{lem:almost_idempotent}
The following statements hold:
\begin{enumerate}
\item $\chi_j\chi_k=0$ for all $j>k+1$; 
\item $u_n \chi_k = \chi_k$ for all $k<n$; 
\item $u_n \chi_k = 0$ for all $k>n+1$. 
\end{enumerate}
\end{lem}
\begin{proof}
\begin{enumerate}
\item Since $\{u_n\}$ is almost idempotent, we know that $u_j u_k = u_k$ for all $j>k$. Then for $j>k+1$ we have 
\begin{align*}
\chi_j^2 \chi_k^2 
&= (u_j-u_{j-1})(u_k-u_{k-1}) 
= (u_j-u_{j-1}) u_k - (u_j-u_{j-1}) u_{k-1} \\
&= u_k-u_k - u_{k-1} + u_{k-1} 
= 0 .
\end{align*}
In particular, $\chi_j^2$ commutes with $\chi_k^2$, and therefore their square roots 
$\chi_j$ and $\chi_k$ also commute and we see that $\chi_j\chi_k = (\chi_j^2\chi_k^2)^{\frac12} = 0$. 

\item For $n>k$ we have 
\[
u_n^2 \chi_k^2 
= u_n^2 (u_k-u_{k-1}) 
= u_k-u_{k-1} 
= \chi_k^2 .
\]
In particular, $u_n^2$ commutes with $\chi_k^2$, which implies that $u_n$ commutes with $\chi_k$ and $u_n\chi_k=\chi_k$. 

\item For $k>n+1$ we have 
\[
u_n \chi_k^2 
= u_n (u_k-u_{k-1}) 
= u_n - u_n 
= 0 . 
\]
In particular, $u_n$ commutes with $\chi_k^2$, which implies that $u_n^{\frac12}$ commutes with $\chi_k$ and $u_n\chi_k = u_n^{\frac12} \big( u_n \chi_k^2 \big)^{\frac12} = 0$. 
\qedhere
\end{enumerate}
\end{proof}

We pick a sequence of elements $v_k\in\{u_n\}_{n\in\N}$ such that $v_k u_{k+1} = u_{k+1}$ for each $k\in\N$ (the simplest choice is of course to take $v_k=u_{k+2}$, but in later sections it will be convenient to choose $v_k=u_{k+3}$ or $v_k=u_{k+4}$, so we allow for this additional flexibility). 
Pick a sequence $\{\alpha_k\}_{k\in\N} \subset (0,\infty)$ of strictly positive numbers, and consider the operators
\begin{align*}
\D_k &:= v_k \D v_k , & 
F_{\alpha_k\D_k} := \alpha_k\D_k (1+\alpha_k^2\D_k^2)^{-\frac12} .
\end{align*}

Since $\{u_n\}_{n\in\N} \subset \A \subset \Lip^*(\D)$, we know from \cref{lem:loc_sa} that (the closure of) $\D_k$ is regular and self-adjoint, and that $\Dom\D$ is a core for $\D_k$. 
In particular, the operator $F_{\alpha_k\D_k}$ is well-defined via continuous functional calculus. 

\begin{defn}
\label{defn:Higson}
For any sequence $\{\alpha_k\}_{k\in\N} \subset (0,\infty)$ of strictly positive numbers, we define the \emph{localised representative} of $\D$ as 
\begin{align*}
\til F_\D(\alpha) := \sum_{k=0}^\infty \chi_k F_{\alpha_k\D_k} \chi_k .
\end{align*}
\end{defn}

\begin{lem}
The operator $\til F_\D(\alpha)$ is well-defined as a strongly convergent series. 
\end{lem}
\begin{proof}
First, since for each $n$ the sum $\sum_{k=0}^\infty \chi_k u_n$ is finite (see \cref{lem:almost_idempotent}), we see that $\til F_\D(\alpha) u_n \psi$ is a finite (hence convergent) sum for each $\psi\in E$. 
Hence $\til F_\D(\alpha)$ converges strongly on the dense subset $\{ u_n \psi \mid n\in\N, \; \psi\in E \}$. 
Second, using the operator inequalities $\pm F_{\alpha_k\D_k} \leq \|F_{\alpha_k\D_k}\| \leq 1$, we see for $K\in\N$ that 
\[
\pm \sum_{k=0}^K \chi_k F_{\alpha_k\D_k} \chi_k \leq \sum_{k=0}^K \chi_k^2 = u_K \leq 1 .
\]
Hence the partial sums are uniformly bounded, and therefore the series converges strongly on all of $E$. 
\end{proof}

\begin{lem}[{\cite[Lemma 2.8]{vdD20_Kasp_open}}]
\label{lem:bdd_transf_rescaled}
Let $\D$ be a regular self-adjoint operator on a Hilbert $B$-module $E$. Let $a\in\End_B(E)$ such that $a(\D\pm i)^{-1}$ is compact. Then for any $\alpha>0$, the operator $a (F_\D - F_{\alpha\D})$ is compact. 
\end{lem}

\begin{thm}
\label{thm:local_rep_KK}
Consider the setting of \cref{ass:half-closed_partition}. 
Then for any $a\in A$, the operator $a(\til F_\D(\alpha)-F_\D)$ is compact. 
Hence $(A,E_B,\til F_D(\alpha))$ is a (bounded) Kasparov $A$-$B$-module, and $[\til F_\D(\alpha)] = [F_{\D}] \in \KK(A,B)$. 
In particular, the class $[\til F_\D(\alpha)]$ is independent of the choices made in the construction. 
\end{thm}
\begin{proof}
Since we have norm-convergence $a u_n \to a$, it suffices to prove the compactness of $u_n (\til F_\D(\alpha)-F_\D)$. 
Recall that $\til F_\D(\alpha) = \sum_{k=0}^\infty \chi_k F_{\alpha_k\D_k} \chi_k$, where $\D_k = v_k\D v_k$. 
By applying \cref{lem:cpt_res_localisation} (using $\chi_k = \chi_k u_{k+1}$ from \cref{lem:almost_idempotent}) we know that $\chi_k (\D_k\pm i)^{-1}$ is compact; hence we can apply \cref{lem:bdd_transf_rescaled} to see that $\chi_k (F_{\alpha_k\D_k} - F_{\D_k})$ is compact. 
We know from \cref{thm:Hilsum} that the commutator $[F_\D,\chi_k] $ is compact. 
Furthermore, applying \cref{lem:local_comparison} (with $a=\chi_k$, $c=u_{k+1}$, and $b=v_k$) we know that $\chi_k(F_{\D_k}-F_{\D})$ is compact. 
From \cref{lem:almost_idempotent} we know that $u_n \til F_\D(\alpha)$ is given by a finite sum, and therefore
\begin{align*}
u_n (\til F_\D(\alpha)-F_{\D}) 
&= \sum_k u_n \chi_k F_{\alpha_k\D_k} \chi_k - u_n F_{\D} 
\stackrel{\ref{lem:bdd_transf_rescaled}}{\sim} \sum_k u_n \big( \chi_k F_{\D_k} \chi_k - \chi_k^2 F_{\D} \big) \\
&\stackrel{\ref{thm:Hilsum}}{\sim} \sum_k u_n \chi_k ( F_{\D_k} - F_{\D} ) \chi_k 
\stackrel{\ref{lem:local_comparison}}{\sim} 0 .
\qedhere
\end{align*}
\end{proof}

\section{Local positivity}
\label{sec:local_positivity}

The aim in this section is to show that the local positivity condition (see \cref{defn:local_positivity_condition} below) implies a \emph{localised} version of the positivity condition in \cref{thm:Connes-Skandalis}. 
This was already proven by the author for the case of first-order differential operators on smooth manifolds \cite[Proposition 3.1]{vdD20_Kasp_open}, and in fact, many of the arguments of \cite[\S3]{vdD20_Kasp_open} can be adapted to the following more abstract context. 

\begin{assumption}
\label{ass:local_positivity}
Let $\D$ be an odd regular symmetric operator on a $\Z_2$-graded Hilbert $C$-module $E$, 
let $\mS$ be an odd regular self-adjoint operator on $E$, 
and let $\chi,\rho,\phi,v\in\End_C(E)$ be even and self-adjoint. 
We define the operator $\D_v := v\D v$.
We assume that the following conditions hold:
\begin{enumlocal}
\item \label{ass:loc_pos_Dom}
$\Dom(\D v) \cap \Ran(v) \subset \Dom(\mS v)$;
\item \label{ass:loc_pos_approx_unit}
$\rho\chi=\chi$, $\phi\rho=\rho$, $v\rho=\rho$, $v\phi=\phi v$, $\|\rho\| \leq 1$, and $\|\phi\| \leq 1$;
\item \label{ass:loc_pos_Lip}
$\rho,\phi\in\Lip(\D)\cap\Lip(\mS)$ and $v\in\Lip^*(\D)\cap\Lip(\mS)$; 
\item \label{ass:loc_pos_cpt_res}
$\chi(1+\D_v^2)^{-\frac12}$ is compact;
\item \label{ass:loc_pos_condition}
there exists a constant $c\in[0,\infty)$ such that for all $\psi\in\Dom(\D v)\cap\Ran(v)$ we have 
\begin{align*}
Q_v(\phi\psi) 
:= \big\la\D v\phi\psi \bigmvert v\mS\phi\psi \big\ra + \big\la v\mS\phi\psi \bigmvert \D v\phi\psi \big\ra 
\geq - c \big\la \phi\psi \bigmvert (1+\D_v^2)^{\frac12}\phi\psi \big\ra .
\end{align*}
\end{enumlocal}
\end{assumption}

Since $v\in\Lip^*(\D)$, we note that $\D_v$ is regular and self-adjoint (see \cref{lem:loc_sa}). Since $\rho,\phi\in\Lip\D$ and $\rho,\phi$ commute with $v$, we know that we also have $\rho,\phi\in\Lip\D_v$. 
Furthermore, since $\phi$ preserves $\Dom(\D v)\cap\Ran(v)$ and we have $\Dom(\D v)\cap\Ran(v) \subset \Dom(\mS v)=\Dom(v\mS)$, we see that $Q_v$ is well-defined. 

\begin{notation}
\label{notation2}
For $\lambda,\mu\in[0,\infty)$ and for $r\in(0,\infty)$, we use the notation 
\begin{align*}
R_{\D_v}^r(\lambda) &:= (r^2+\lambda+\D_v^2)^{-1} , & 
R_\mS(\mu) &:= (1+\mu+\mS^2)^{-1} . 
\end{align*}
We consider the following bounded operators: 
\begin{align*}
k_{\D_v}^r(\lambda) &:= \sqrt{r^2+\lambda} R_{\D_v}^r(\lambda), & 
k_\mS(\mu) &:= \sqrt{1+\mu} R_\mS(\mu) , \\
h_{\D_v}^r(\lambda) &:= \D_v R_{\D_v}^r(\lambda) , & 
h_\mS(\mu) &:= \mS R_\mS(\mu) . 
\end{align*}
We also define the operator 
\begin{align*}
B(\lambda) &:= [\mS,\rho] \phi h_{\D_v}^r(\lambda) \rho .
\end{align*}
We now \emph{redefine} the operators $M_m(\lambda,\mu)$, using $\D_v$ instead of $\D$, and inserting $\rho$: 
\begin{align*}
M_1(\lambda,\mu) &:= h_{\D_v}^r(\lambda) \rho h_\mS(\mu) , & 
M_2(\lambda,\mu) &:= k_{\D_v}^r(\lambda) \rho h_\mS(\mu) , \\
M_3(\lambda,\mu) &:= h_{\D_v}^r(\lambda) \rho k_\mS(\mu) , & 
M_4(\lambda,\mu) &:= k_{\D_v}^r(\lambda) \rho k_\mS(\mu) .
\end{align*}
Furthermore, we define
\begin{align*}
\hat B(\lambda,\mu) &:= 2 \Re \big\la B(\lambda) k_\mS(\mu) \chi \psi \bigmvert k_\mS(\mu) \chi \psi \big\ra + 2 \Re \big\la B(\lambda) h_\mS(\mu) \chi \psi \bigmvert h_\mS(\mu) \chi \psi \big\ra , \\
\hat M(\lambda,\mu) &:= \sum_{m=1}^4 2 \Re \big\la [\phi,\D] v M_m(\lambda,\mu) \chi \psi \bigmvert v \mS \phi M_m(\lambda,\mu) \chi \psi \big\ra , \\
\hat Q(\lambda,\mu) &:= \sum_{m=1}^4 Q_v\big( \phi M_m(\lambda,\mu) \chi \psi \big) . 
\end{align*}
\end{notation}

\begin{lem}[{\cite[Lemma 3.4]{KS19}}]
\label{lem:resolvent_Dv2}
For any $r\in\R$ with $|r| > \big\|[\D,v] v\big\|$, 
the operator 
\[
ir+\D v^2\colon\Dom(\D v^2)\to E 
\]
is bijective, its inverse $(ir+\D v^2)^{-1}$ is adjointable on $E$, and we have the equality
\[
(ir+\D_v)^{-1} v = v (ir+\D v^2)^{-1} . 
\]
\end{lem}

\begin{lem}[cf.\ {\cite[Lemma 6.15]{KS19}}]
\label{lem:range_k_h}
For any $r > \big\|[\D,v] v\big\|$,
we have the following inclusions:
\begin{align*}
\Ran\big(k_{\D_v}^r(\lambda) v\big) &\subset \Dom(\D v)\cap\Ran(v) , &
\Ran\big(h_{\D_v}^r(\lambda) v\big) &\subset \Dom(\D v)\cap\Ran(v) . 
\end{align*}
\end{lem}
\begin{proof}
From \cref{lem:resolvent_Dv2} we know for any $r > \big\|[\D,v] v\big\|$ that 
\[
\Ran\big( (ir+\D_v)^{-1} v \big) = v\cdot\Dom(\D v^2) = \Dom(\D v) \cap \Ran(v) . 
\]
The inclusions then follow, because we can rewrite 
\begin{align*}
R_{\D_v}^r(\lambda) v 
&= \big( i\sqrt{r^2+\lambda} + \D_v \big)^{-1} v \big( -i\sqrt{r^2+\lambda} + \D v^2 \big)^{-1} , \\
\D_v R_{\D_v}^r(\lambda) v 
&= \big( -i\sqrt{r^2+\lambda} + \D_v \big)^{-1} v - i\sqrt{r^2+\lambda} R_{\D_v}^r(\lambda) v . 
\qedhere
\end{align*}
\end{proof}

We will study the positivity (modulo compact operators) of the operator $\chi [F_{\D_v},F_\mS] \chi$. 
Using $\rho\chi=\chi$, we can rewrite 
\begin{align*}
\chi [F_{\D_v},F_\mS] \chi 
&= \chi (\rho F_{\D_v} F_\mS + F_\mS F_{\D_v} \rho)\chi \\
&= \chi (F_{\D_v}\rho F_\mS + F_\mS\rho F_{\D_v})\chi + \chi ([\rho,F_{\D_v}] F_\mS + F_\mS [F_{\D_v},\rho])\chi .
\end{align*}
We note that $[F_{\D_v},\rho]\chi$ is compact by \cref{lem:comm_cpt}.\ref{lem:comm_cpt_F} (using $\rho\in\Lip(\D_v)$ and condition \ref{ass:loc_pos_cpt_res}), and therefore it suffices to consider instead the operator $\chi (F_{\D_v}\rho F_\mS + F_\mS\rho F_{\D_v})\chi$. 
Furthermore, consider for any $0\neq r\in\R$ the operator 
\[
F_{\D_v}^r := \D_v (r^2+\D_v^2)^{-\frac12} . 
\]
Since the function $x\mapsto x(1+x^2)^{-\frac12} - x(r^2+x^2)^{-\frac12}$ lies in $C_0(\R)$, we know that also $(F_{\D_v}-F_{\D_v}^r)\chi$ is compact. Hence we may replace $F_{\D_v}$ by $F_{\D_v}^r$. 
Applying \cref{lem:integral_formula} twice, we then rewrite 
\begin{align}
\label{eq:comm_integral}
\big\la \psi \bigmvert \chi (F_{\D_v}^r\rho F_\mS + F_\mS\rho F_{\D_v}^r) \chi \psi \big\ra 
&= 2 \Re \big\la \chi \psi \bigmvert F_{\D_v}^r\rho F_\mS \chi \psi \big\ra \\
&= \frac{1}{\pi^2} \int_0^\infty \int_0^\infty (\lambda\mu)^{-\frac12} 2 \Re \big\la \chi \psi \bigmvert \D_v R_{\D_v}^r(\lambda) \rho \mS R_\mS(\mu) \chi \psi \big\ra d\lambda d\mu . \nonumber
\end{align}
Our first task is to study the integrand on the right-hand-side. Via a straightforward but somewhat tedious calculation, we will rewrite this integrand in terms of the operators defined above. 

\begin{lem}
\label{lem:positivity_local_expr}
For any $\psi\in E$ we have 
\begin{align*}
2 \Re \big\la \chi \psi \bigmvert \D_v R_{\D_v}^r(\lambda) \rho \mS R_\mS(\mu) \chi \psi \big\ra &= \hat B(\lambda,\mu) + \hat M(\lambda,\mu) + \hat Q(\lambda,\mu) .
\end{align*}
\end{lem}
\begin{proof}
We compute:
\begin{align}
\label{eq:local_expr}
&\big\la \chi \psi \bigmvert \D_v R_{\D_v}^r(\lambda) \rho \mS R_\mS(\mu) \chi \psi \big\ra \nonumber\\
&= \big\la \rho (1+\mu+\mS^2)R_\mS(\mu) \chi \psi \bigmvert \D_v R_{\D_v}^r(\lambda) \rho \mS R_\mS(\mu) \chi \psi \big\ra \nonumber\\
&= \big\la \rho k_\mS(\mu) \chi \psi \bigmvert \D_v R_{\D_v}^r(\lambda) \phi^2\rho \mS k_\mS(\mu) \chi \psi \big\ra 
+ \big\la \phi^2\rho \mS h_\mS(\mu) \chi \psi \bigmvert \D_v R_{\D_v}^r(\lambda) \rho h_\mS(\mu) \chi \psi \big\ra \nonumber\\
&= \big\la \rho k_\mS(\mu) \chi \psi \bigmvert \D_v R_{\D_v}^r(\lambda) \phi[\phi\rho,\mS] k_\mS(\mu) \chi \psi \big\ra 
+ \big\la \rho k_\mS(\mu) \chi \psi \bigmvert \D_v R_{\D_v}^r(\lambda) \phi \mS \phi\rho k_\mS(\mu) \chi \psi \big\ra \nonumber\\
&\quad+ \big\la \phi [\phi\rho,\mS] h_\mS(\mu) \chi \psi \bigmvert \D_v R_{\D_v}^r(\lambda) \rho h_\mS(\mu) \chi \psi \big\ra 
+ \big\la \mS \phi\rho h_\mS(\mu) \chi \psi \bigmvert \phi \D_v R_{\D_v}^r(\lambda) \rho h_\mS(\mu) \chi \psi \big\ra \nonumber\\
&= \big\la B(\lambda) k_\mS(\mu) \chi \psi \bigmvert k_\mS(\mu) \chi \psi \big\ra 
+ \big\la h_\mS(\mu) \chi \psi \bigmvert B(\lambda) h_\mS(\mu) \chi \psi \big\ra \nonumber\\
&\quad+ \big\la \phi \D_v R_{\D_v}^r(\lambda) \rho k_\mS(\mu) \chi \psi \bigmvert \mS \phi\rho k_\mS(\mu) \chi \psi \big\ra 
+ \big\la \mS \phi\rho h_\mS(\mu) \chi \psi \bigmvert \phi \D_v R_{\D_v}^r(\lambda) \rho h_\mS(\mu) \chi \psi \big\ra ,
\end{align}
where we have inserted the definition of $B(\lambda)$. 
To rewrite the terms on the last line, we consider either $\xi = \rho k_\mS(\mu) \chi \psi$ or $\xi = \rho h_\mS(\mu) \chi \psi$. 
Since $v\rho=\rho$, we note that in both cases we have $\xi = v\xi \in\Ran(v)$, so from \cref{lem:range_k_h} and condition \ref{ass:loc_pos_Dom} we know that $k_{\D_v}^r(\lambda)\xi$ and $h_{\D_v}^r(\lambda)\xi$ lie in $\Dom(v\mS)$. 
Since $\phi$ preserves $\Dom(v\mS)$, we can compute
\begin{align*}
&\big\la \phi \D_v R_{\D_v}^r(\lambda) \xi \bigmvert \mS \phi \xi \big\ra \\
&= \big\la \phi \D v R_{\D_v}^r(\lambda) \xi \bigmvert v\mS \phi (r^2+\lambda+\D_v^2)R_{\D_v}^r(\lambda) \xi \big\ra \\
&= \big\la \phi \D v k_{\D_v}^r(\lambda) \xi \bigmvert v\mS \phi k_{\D_v}^r(\lambda) \xi \big\ra 
+ \big\la v\mS \phi h_{\D_v}^r(\lambda) \xi \bigmvert \phi \D v h_{\D_v}^r(\lambda) \xi \big\ra \\
&= \big\la [\phi,\D] v k_{\D_v}^r(\lambda) \xi \bigmvert v\mS \phi k_{\D_v}^r(\lambda) \xi \big\ra 
+ \big\la \D\phi v k_{\D_v}^r(\lambda) \xi \bigmvert v\mS \phi k_{\D_v}^r(\lambda) \xi \big\ra \\
&\quad+ \big\la v\mS \phi h_{\D_v}^r(\lambda) \xi \bigmvert [\phi,\D] v h_{\D_v}^r(\lambda) \xi \big\ra 
+ \big\la v\mS \phi h_{\D_v}^r(\lambda) \xi \bigmvert \D \phi v h_{\D_v}^r(\lambda) \xi \big\ra .
\end{align*}
Taking twice the real part, and inserting the definition of $Q_v$, this yields
\begin{multline*}
2 \Re \big\la \phi \D_v R_{\D_v}^r(\lambda) \xi \bigmvert \mS \phi \xi \big\ra 
= 2 \Re \big\la [\phi,\D] v k_{\D_v}^r(\lambda) \xi \bigmvert v\mS \phi k_{\D_v}^r(\lambda) \xi \big\ra \\
+ 2 \Re \big\la [\phi,\D] v h_{\D_v}^r(\lambda) \xi \bigmvert v\mS \phi h_{\D_v}^r(\lambda) \xi \big\ra 
+ Q_v\big( \phi k_{\D_v}^r(\lambda) \xi \big) 
+ Q_v\big( \phi h_{\D_v}^r(\lambda) \xi \big) . 
\end{multline*}
Inserting the latter into \cref{eq:local_expr}, we thus obtain 
\begin{align*}
&2 \Re \big\la \chi \psi \bigmvert \D_v R_{\D_v}^r(\lambda) \rho \mS R_\mS(\mu) \chi \psi \big\ra \\
&= 2 \Re \big\la B(\lambda) k_\mS(\mu) \chi \psi \bigmvert k_\mS(\mu) \chi \psi \big\ra 
+ 2 \Re \big\la h_\mS(\mu) \chi \psi \bigmvert B(\lambda) h_\mS(\mu) \chi \psi \big\ra \\
&\quad+ 2 \Re \big\la [\phi,\D] v k_{\D_v}^r(\lambda) \rho k_\mS(\mu) \chi \psi \bigmvert v\mS \phi k_{\D_v}^r(\lambda) \rho k_\mS(\mu) \chi \psi \big\ra \\
&\quad+ 2 \Re \big\la [\phi,\D] v h_{\D_v}^r(\lambda) \rho k_\mS(\mu) \chi \psi \bigmvert v\mS \phi h_{\D_v}^r(\lambda) \rho k_\mS(\mu) \chi \psi \big\ra \\
&\quad+ Q_v\big( \phi k_{\D_v}^r(\lambda) \rho k_\mS(\mu) \chi \psi \big) 
+ Q_v\big( \phi h_{\D_v}^r(\lambda) \rho k_\mS(\mu) \chi \psi \big) \\
&\quad+ 2 \Re \big\la [\phi,\D] v k_{\D_v}^r(\lambda) \rho h_\mS(\mu) \chi \psi \bigmvert v\mS \phi k_{\D_v}^r(\lambda) \rho h_\mS(\mu) \chi \psi \big\ra \\
&\quad+ 2 \Re \big\la [\phi,\D] v h_{\D_v}^r(\lambda) \rho h_\mS(\mu) \chi \psi \bigmvert v\mS \phi h_{\D_v}^r(\lambda) \rho h_\mS(\mu) \chi \psi \big\ra \\
&\quad+ Q_v\big( \phi k_{\D_v}^r(\lambda) \rho h_\mS(\mu) \chi \psi \big) 
+ Q_v\big( \phi h_{\D_v}^r(\lambda) \rho h_\mS(\mu) \chi \psi \big) \\
&= 2 \Re \big\la B(\lambda) k_\mS(\mu) \chi \psi \bigmvert k_\mS(\mu) \chi \psi \big\ra 
+ 2 \Re \big\la B(\lambda) h_\mS(\mu) \chi \psi \bigmvert h_\mS(\mu) \chi \psi \big\ra \\
&\quad+ \sum_{m=1}^4 2 \Re \big\la [\phi,\D] v M_m(\lambda,\mu) \chi \psi \bigmvert v\mS \phi M_m(\lambda,\mu) \chi \psi \big\ra 
+ \sum_{m=1}^4 Q_v\big( \phi M_m(\lambda,\mu) \chi \psi \big) . 
\qedhere
\end{align*}
\end{proof}

\begin{lem}
\label{lem:B}
We have the inequality 
\begin{align*}
\pm \frac{1}{\pi^2} \int_0^\infty \int_0^\infty (\lambda\mu)^{-\frac12} \hat B(\lambda,\mu) d\lambda d\mu &\leq C_B \big\la \chi \psi \bigmvert \chi \psi \big\ra ,
\end{align*}
where $C_B := 2 \big\| [\mS,\rho] \big\|$. 
\end{lem}
\begin{proof}
We note that the integral over $\lambda$ of $\pi^{-1} \lambda^{-\frac12} \D_v R_{\D_v}^r(\lambda)$ converges strongly to $F_{\D_v}^r = \D_v(r^2+\D_v^2)^{-\frac12} \leq 1$, and thus we have the operator inequality 
\begin{align*}
\pm \frac{1}{\pi} \int_0^\infty \lambda^{-\frac12} \big( B(\lambda) + B(\lambda)^* \big) d\lambda 
&= \pm [\mS,\rho] \phi F_{\D_v}^r \rho \pm \rho F_{\D_v}^r \phi [\rho,\mS] 
\leq 2 \big\| [\mS,\rho] \big\| = C_B . 
\end{align*}
Inserting this into the definition of $\hat B(\lambda,\mu)$, the remaining integral over $\mu$ is norm-con\-ver\-gent, and using $k_\mS(\mu)^2+h_\mS(\mu)^2=R_\mS(\mu)$ we find 
\begin{align*}
&\pm \frac{1}{\pi^2} \int_0^\infty \int_0^\infty (\lambda\mu)^{-\frac12} \hat B(\lambda,\mu) d\lambda d\mu \\
&\quad\leq \frac{1}{\pi} C_B \int_0^\infty \mu^{-\frac12} \big\la k_\mS(\mu) \chi \psi \bigmvert k_\mS(\mu) \chi \psi \big\ra d\mu 
+ \frac{1}{\pi} C_B \int_0^\infty \mu^{-\frac12} \big\la h_\mS(\mu) \chi \psi \bigmvert h_\mS(\mu) \chi \psi \big\ra d\mu \\
&\quad= \frac{1}{\pi} C_B \int_0^\infty \mu^{-\frac12} \big\la \chi \psi \bigmvert R_\mS(\mu) \chi \psi \big\ra d\mu 
= C_B \big\la \chi \psi \bigmvert (1+\mS^2)^{-\frac12} \chi \psi \big\ra \leq C_B \big\la \chi \psi \bigmvert \chi \psi \big\ra . 
\qedhere
\end{align*}
\end{proof}

The following lemma generalises \cref{lem:integral_estimate_first-order}. 

\begin{lem}
\label{lem:integral_estimate_first-order2}
Let $T$ be any densely defined symmetric operator on $E$ such that $T(r^2+\D_v^2)^{-\frac12}$ is bounded. 
For $\psi\in E$, we have the inequality 
\begin{align*}
\sum_{m=1}^4 \frac{1}{\pi^2} \int_0^\infty \int_0^\infty (\mu\lambda)^{-\frac12} \big\la M_m(\lambda,\mu)\psi \bigmvert T M_m(\lambda,\mu)\psi \big\ra d\lambda d\mu 
\leq \big\| T(r^2+\D_v^2)^{-\frac12} \big\| \, \|\rho\|^2 \, \la\psi|\psi\ra . 
\end{align*}
\end{lem}
\begin{proof}
Let us consider the four integrals (for $m=1,2,3,4$) given by 
\[
I_m := \frac{1}{\pi^2} \int_0^\infty \int_0^\infty (\mu\lambda)^{-\frac12} \big\la M_m(\lambda,\mu) \psi \bigmvert T M_m(\lambda,\mu) \psi \big\ra d\lambda d\mu .
\]
Since $T(r^2+\D_v^2)^{-\frac12}$ is bounded and $T$ is symmetric, we know from \cref{lem:interpolation} that also $(r^2+\D_v^2)^{-\frac14}T(r^2+\D_v^2)^{-\frac14}$ is bounded, and 
\[
\big\| (r^2+\D_v^2)^{-\frac14}T(r^2+\D_v^2)^{-\frac14} \big\| \leq \big\| T(r^2+\D_v^2)^{-\frac12} \big\| . 
\]
We obtain the operator inequality 
\begin{align*}
h_{\D_v}^r(\lambda) T h_{\D_v}^r(\lambda) &= (r^2+\D_v^2)^{\frac14} h_{\D_v}^r(\lambda) \, (r^2+\D_v^2)^{-\frac14} T (r^2+\D_v^2)^{-\frac14} \, h_{\D_v}^r(\lambda) (r^2+\D_v^2)^{\frac14} \\
& \leq (r^2+\D_v^2)^{\frac14} h_{\D_v}^r(\lambda) \, \big\| (r^2+\D_v^2)^{-\frac14} T (r^2+\D_v^2)^{-\frac14} \big\| \, h_{\D_v}^r(\lambda) (r^2+\D_v^2)^{\frac14} \\
&\leq \big\| T (r^2+\D_v^2)^{-\frac12} \big\| \, (r^2+\D_v^2)^{\frac12} h_{\D_v}^r(\lambda)^2 . 
\end{align*}
Similarly, we also have the operator inequality 
\[
k_{\D_v}^r(\lambda) T k_{\D_v}^r(\lambda) \leq \big\| T (r^2+\D_v^2)^{-\frac12} \big\| \, (r^2+\D_v^2)^{\frac12} k_{\D_v}^r(\lambda)^2 .
\]
We note that $h_{\D_v}^r(\lambda)^2 + k_{\D_v}^r(\lambda)^2 = R_{\D_v}^r(\lambda)$. 
Moreover, by \cref{lem:integral_formula}, the integral of $\lambda^{-\frac12} (r^2+\D_v^2)^{\frac12} R_{\D_v}^r(\lambda)$ converges strongly to $\pi$, and we find 
\[
\frac{1}{\pi} \int_0^\infty \lambda^{-\frac12} \big( h_{\D_v}^r(\lambda) T h_{\D_v}^r(\lambda) + k_{\D_v}^r(\lambda) T k_{\D_v}^r(\lambda) \big) d\lambda \leq \big\| T (r^2+\D_v^2)^{-\frac12} \big\| .
\]
Thus we obtain the inequalities 
\begin{align*}
I_1 + I_2 &\leq \frac{1}{\pi} \big\| T (r^2+\D_v^2)^{-\frac12} \big\| \, \|\rho\|^2 \, \int_0^\infty \mu^{-\frac12} \big\la h_\mS(\mu) \psi \bigmvert h_\mS(\mu) \psi \big\ra d\mu , \\
I_3 + I_4 &\leq \frac{1}{\pi} \big\| T (r^2+\D_v^2)^{-\frac12} \big\| \, \|\rho\|^2 \, \int_0^\infty \mu^{-\frac12} \big\la k_\mS(\mu) \psi \bigmvert k_\mS(\mu) \psi \big\ra d\mu . 
\end{align*}
Summing up the latter two inequalities and computing the remaining norm-convergent integral over $\mu$, we obtain
\begin{align*}
\sum_{m=1}^4 I_m &\leq \frac{1}{\pi} \big\| T (r^2+\D_v^2)^{-\frac12} \big\| \, \|\rho\|^2 \, \int_0^\infty \mu^{-\frac12} \big\la \psi \bigmvert R_\mS(\mu) \psi \big\ra d\mu \\
&= \big\| T (r^2+\D_v^2)^{-\frac12} \big\| \, \|\rho\|^2 \, \big\la \psi \bigmvert (1+\mS^2)^{-\frac12} \psi \big\ra \\
&\leq \big\| T (r^2+\D_v^2)^{-\frac12} \big\| \, \|\rho\|^2 \, \big\la \psi \bigmvert \psi \big\ra . 
\qedhere
\end{align*}
\end{proof}

\begin{lem}
\label{lem:Q}
We have the inequality 
\begin{align*}
\frac{1}{\pi^2} \int_0^\infty \int_0^\infty (\lambda\mu)^{-\frac12} \hat Q(\lambda,\mu) d\lambda d\mu 
&\geq - C_Q \, \big\la \chi \psi \bigmvert \chi \psi \big\ra , 
\end{align*}
where $C_Q := c \big\| \phi (1+\D_v^2)^{\frac12} \phi (r^2+\D_v^2)^{-\frac12} \big\|$. 
\end{lem}
\begin{proof}
By condition \ref{ass:loc_pos_condition} in \cref{ass:local_positivity}, we have the inequality 
\[
Q_v\big( \phi M_m(\lambda,\mu) \chi \psi \big) \geq - c \big\la M_m(\lambda,\mu) \chi \psi \bigmvert \phi (1+\D_v^2)^{\frac12} \phi M_m(\lambda,\mu) \chi \psi \big\ra . 
\]
The statement then follows by applying \cref{lem:integral_estimate_first-order2} with the operator $T = \phi (1+\D_v^2)^{\frac12} \phi$. 
\end{proof}

It remains to consider the term $\hat M(\lambda,\mu)$ in \cref{lem:positivity_local_expr}. 
Of course, in the special case where $\phi=\Id$, we have $\hat M(\lambda,\mu) = 0$, and then we obtain the following result.

\begin{prop}
\label{prop:local_positivity_phi=1}
Consider the setting of \cref{ass:local_positivity}.
Suppose furthermore that $\phi=\Id$. 
Then for any $0<\kappa<2$ there exists an $\alpha>0$ such that the operator $\chi [ F_{\D_v} , F_{\alpha \mS} ] \chi + \kappa \chi^2$ is positive modulo compact operators:
\[
\chi [F_{\D_v} , F_{\alpha \mS}] \chi \gtrsim - \kappa \chi^2 .
\]
\end{prop}
\begin{proof}
As mentioned at the start of this section, we know that $\chi [F_{\D_v} , F_\mS] \chi$ is equal to $\chi \big( F_{\D_v}^r \rho F_\mS + F_\mS \rho F_{\D_v}^r \big) \chi$ modulo compact operators. 
For any $\psi\in E$, we thus need to estimate the integral in \cref{eq:comm_integral}. 
From \cref{lem:positivity_local_expr} we obtain the equality 
\begin{align*}
\big\la \psi \bigmvert \chi \big( F_{\D_v}^r \rho F_\mS + F_\mS \rho F_{\D_v}^r \big) \chi \psi \big\ra 
&= \frac{1}{\pi^2} \int_0^\infty \int_0^\infty (\lambda\mu)^{-\frac12} \big( \hat B(\lambda,\mu) + \hat M(\lambda,\mu) + \hat Q(\lambda,\mu) \big) d\lambda d\mu . 
\end{align*}
Since $\phi=\Id$, we have $\hat M(\lambda,\mu)=0$. 
Using \cref{lem:B,lem:Q}, we then obtain the estimate 
\begin{align*}
\big\la \psi \bigmvert \chi \big( F_{\D_v}^r \rho F_{\mS} + F_{\mS} \rho F_{\D_v}^r \big) \chi \psi \big\ra \geq - ( C_B + C_Q ) \lla\chi\psi\rra . 
\end{align*}
Since this holds for any $\psi$, we have the operator inequality
\[
\chi \big( F_{\D_v}^r \rho F_{\mS} + F_{\mS} \rho F_{\D_v}^r \big) \chi \geq - ( C_B + C_Q ) \chi^2 .
\]
We have therefore shown that 
\[
\chi [F_{\D_v} , F_{\mS}] \chi \gtrsim - ( C_B + C_Q ) \chi^2 .
\]
Next, we fix $0<\kappa<2$. 
If we replace $\mS$ by $\alpha \mS$ for some $\alpha>0$, then the constants $C_B$ and $C_Q$ are replaced by $\alpha C_B$ and $\alpha C_Q$, respectively. Indeed, for $C_B = 2 \big\| [\mS,\rho] \big\|$ this is obvious, and we note that $C_Q$ is proportional to the constant $c$ from condition \ref{ass:loc_pos_condition} in \cref{ass:local_positivity}. 
Thus, by choosing $\alpha$ small enough, we can ensure that $\alpha ( C_B + C_Q ) < \kappa < 2$. 
\end{proof}

In the following, we provide two different sufficient conditions which allow us to deal with the term $\hat M(\lambda,\mu)$ also for non-trivial $\phi$. 

\begin{prop}
\label{prop:local_positivity_strong}
In addition to the setting of \cref{ass:local_positivity}, suppose furthermore that the following strengthening of condition \ref{ass:loc_pos_condition} is satisfied:
\begin{enumlocal}[label={(L\arabic*')}]
\addtocounter{enumlocali}{-1}
\item \label{ass:strong_loc_pos_condition}
there exist constants $\nu\in(0,\infty)$ and $c\in[0,\infty)$ such that for all $\psi\in\Dom(\D v)\cap\Ran(v)$ we have 
\begin{align}
Q_v(\phi\psi) \geq \nu \big\lla v\mS\phi\psi\big\rra - c \big\la \phi\psi \bigmvert (1+\D_v^2)^{\frac12}\phi\psi \big\ra .
\end{align}
\end{enumlocal}
Then for any $0<\kappa<2$ there exists an $\alpha>0$ such that the operator $\chi [ F_{\D_v} , F_{\alpha \mS} ] \chi + \kappa \chi^2$ is positive modulo compact operators:
\[
\chi [F_{\D_v} , F_{\alpha \mS}] \chi \gtrsim - \kappa \chi^2 .
\]
\end{prop}
\begin{proof}
First of all, for any $\beta>0$ we can estimate 
\begin{align*}
\pm \hat M(\lambda,\mu) &= \pm \sum_{m=1}^4 2 \Re \big\la \beta^{-1}[\phi,\D] v M_m(\lambda,\mu) \chi \psi \bigmvert \beta v \mS \phi M_m(\lambda,\mu) \chi \psi \big\ra \\
&\leq \sum_{m=1}^4 \beta^2 \big\lla v \mS \phi M_m(\lambda,\mu) \chi \psi \big\rra + \beta^{-2} \big\lla [\phi,\D] v M_m(\lambda,\mu) \chi \psi \big\rra .
\end{align*}
Combined with condition \ref{ass:strong_loc_pos_condition} this yields 
\begin{align*}
\hat M(\lambda,\mu) + \hat Q(\lambda,\mu) 
&\geq (\nu - \beta^2) \big\lla v \mS \phi M_m(\lambda,\mu) \chi \psi \big\rra - \beta^{-2} \big\lla [\phi,\D] v M_m(\lambda,\mu) \chi \psi \big\rra \\
&\quad- c \big\la \phi M_m(\lambda,\mu) \chi \psi \bigmvert (1+\D_v^2)^{\frac12}\phi M_m(\lambda,\mu) \chi \psi \big\ra .
\end{align*}
Taking the double integral of this inequality, and estimating the second and third terms by applying \cref{lem:integral_estimate_first-order2} with $T = \beta^{-2} v [\D,\phi][\phi,\D] v + c \phi(1+\D_v^2)^{\frac12}\phi$, we obtain 
\begin{align}
\label{eq:estimate_strong}
\frac{1}{\pi^2} \int_0^\infty \int_0^\infty (\lambda\mu)^{-\frac12} \big( \hat M(\lambda,\mu) + \hat Q(\lambda,\mu) \big) d\lambda d\mu 
\geq (\nu-\beta^2) P - C_M(\beta,c) \lla\chi\psi\rra , 
\end{align}
where we define 
\begin{align*}
P &:= \frac{1}{\pi^2} \int_0^\infty \int_0^\infty (\lambda\mu)^{-\frac12} \big\lla v \mS \phi M_m(\lambda,\mu) \chi \psi \big\rra d\lambda d\mu \geq 0 , \\
C_M(\beta,c) &:= \beta^{-2} \big\| v [\D,\phi] [\phi,\D] v (r^2+\D_v^2)^{-\frac12} \big\| + c \big\| \phi (1+\D_v^2)^{\frac12} \phi (r^2+\D_v^2)^{-\frac12} \big\| .
\end{align*}
As in the proof of \cref{prop:local_positivity_phi=1}, we need to consider the integral 
\begin{align*}
\big\la \psi \bigmvert \chi \big( F^r_{\D_v} \rho F_\mS + F_\mS \rho F^r_{\D_v} \big) \chi \psi \big\ra 
&= \frac{1}{\pi^2} \int_0^\infty \int_0^\infty (\lambda\mu)^{-\frac12} \big( \hat B(\lambda,\mu) + \hat M(\lambda,\mu) + \hat Q(\lambda,\mu) \big) d\lambda d\mu . 
\end{align*}
Combining \cref{eq:estimate_strong} with \cref{lem:B}, we obtain the estimate 
\begin{align}
\label{eq:estimate_strong2}
\big\la \psi \bigmvert \chi \big( F^r_{\D_v} \rho F_\mS + F_\mS \rho F^r_{\D_v} \big) \chi \psi \big\ra \geq (\nu-\beta^2) P - \big( C_B + C_M(\beta,c) \big) \lla\chi\psi\rra . 
\end{align}
We now replace $\mS$ by $\alpha \mS$ for some $\alpha\in(0,\infty)$. 
We recall that $C_B$ is then replaced by $\alpha C_B$. 
Moreover, we note that condition \ref{ass:strong_loc_pos_condition} implies 
\[
\big\la \D_v \phi\psi \bigmvert \alpha \mS\phi\psi \big\ra + \big\la \alpha \mS\phi\psi \bigmvert \D_v \phi\psi \big\ra 
\geq \alpha^{-1} \nu \big\la \alpha v\mS \phi \psi \bigmvert \alpha v\mS \phi \psi \big\ra - \alpha c \big\la \phi\psi \bigmvert (1+\D_v^2)^{\frac12}\phi\psi \big\ra ,
\]
and hence $\nu$ and $c$ are replaced by $\alpha^{-1} \nu$ and $\alpha c$. 
We choose $\beta$ large enough such that 
\[
\beta^{-2} \big\| v [\D,\phi] [\phi,\D] v (r^2+\D_v^2)^{-\frac12} \big\| < \frac13 \kappa .
\]
We then choose $\alpha$ small enough such that 
\begin{align*}
\alpha^{-1} \nu - \beta^2 &> 0 , & 
\alpha c \big\| \phi (1+\D_v^2)^{\frac12} \phi (r^2+\D_v^2)^{-\frac12} \big\| &< \frac13\kappa , & 
\alpha C_B &< \frac13\kappa . 
\end{align*}
These choices ensure that $\alpha C_B + C_M(\beta,\alpha c) < \kappa$. 
Thus from \cref{eq:estimate_strong2} we obtain 
\begin{align*}
\big\la \psi \bigmvert \chi \big( F^r_{\D_v} \rho F_{\alpha \mS} + F_{\alpha \mS} \rho F^r_{\D_v} \big) \chi \psi \big\ra \geq - \kappa \lla\chi\psi\rra . 
\end{align*}
Since this holds for any $\psi$, we have the operator inequality
\[
\chi \big( F^r_{\D_v} \rho F_{\alpha \mS} + F_{\alpha \mS} \rho F^r_{\D_v} \big) \chi \geq - \kappa \chi^2 .
\]
Since $\chi [F_{\D_v} , F_{\alpha \mS}] \chi$ is equal to $\chi \big( F^r_{\D_v} \rho F_{\alpha \mS} + F_{\alpha \mS} \rho F^r_{\D_v} \big) \chi$ modulo compact operators, this completes the proof.
\end{proof}

\begin{prop}
\label{prop:local_positivity_C2}
In addition to the setting of \cref{ass:local_positivity}, suppose furthermore that the following condition is satisfied:
\begin{enumlocal}
\item \label{ass:weak_C2}
The operators $\phi$ and $\phi[\D_v,\phi]$ map $\Dom(\D_v)$ to $\Dom \mS$. 
\end{enumlocal}
Then for any $0<\kappa<2$ there exists an $\alpha>0$ such that the operator $\chi [ F_{\D_v} , F_{\alpha \mS} ] \chi + \kappa \chi^2$ is positive modulo compact operators:
\[
\chi [F_{\D_v} , F_{\alpha \mS}] \chi \gtrsim - \kappa \chi^2 .
\]
\end{prop}
\begin{proof}
First, we will derive an estimate for the double integral of $\hat M(\lambda,\mu)$. 
Since by assumption $\phi$ maps $\Dom(\D_v)$ to $\Dom(\mS)$, we know that $\mS\phi(r^2+\D_v^2)^{-\frac12}$ is bounded, and therefore also $[\D_v,\phi]\mS\phi(r^2+\D_v^2)^{-\frac12}$ is bounded. 
Furthermore, we have assumed that $\phi[\D_v,\phi]$ also maps $\Dom(\D_v)$ to $\Dom(\mS)$, which ensures that $\phi \mS[\phi,\D_v](r^2+\D_v^2)^{-\frac12} = \mS\phi [\phi,\D_v](r^2+\D_v^2)^{-\frac12} + [\phi,\mS] [\phi,\D_v](r^2+\D_v^2)^{-\frac12}$ is bounded as well. 
By applying \cref{lem:integral_estimate_first-order2} with the symmetric operator $T = [\D_v,\phi]\mS\phi + \phi \mS[\phi,\D_v]$, we obtain the inequality 
\begin{align*}
\pm \frac{1}{\pi^2} \int_0^\infty \int_0^\infty (\lambda\mu)^{-\frac12} \hat M(\lambda,\mu) d\lambda d\mu &\leq C_M \big\la \chi \psi \bigmvert \chi \psi \big\ra , 
\end{align*}
where $C_M := \big\| \big( [\D_v,\phi]\mS\phi + \phi \mS[\phi,\D_v] \big) (r^2+\D_v^2)^{-\frac12} \big\|$. 
As in the proof of \cref{prop:local_positivity_phi=1}, we need to consider the integral 
\begin{align*}
\big\la \psi \bigmvert \chi \big( F^r_{\D_v} \rho F_\mS + F_\mS \rho F^r_{\D_v} \big) \chi \psi \big\ra 
&= \frac{1}{\pi^2} \int_0^\infty \int_0^\infty (\lambda\mu)^{-\frac12} \big( \hat B(\lambda,\mu) + \hat M(\lambda,\mu) + \hat Q(\lambda,\mu) \big) d\lambda d\mu . 
\end{align*}
Combining \cref{lem:B,lem:Q} with the above inequality for $\hat M(\lambda,\mu)$, we now obtain the estimate 
\begin{align*}
\big\la \psi \bigmvert \chi \big( F^r_{\D_v} \rho F_\mS + F_\mS \rho F^r_{\D_v} \big) \chi \psi \big\ra \geq - ( C_B + C_Q + C_M ) \lla\chi\psi\rra . 
\end{align*}
The proof then proceeds exactly as in \cref{prop:local_positivity_phi=1}. 
\end{proof}

\begin{remark}
The condition \ref{ass:weak_C2} is quite naturally satisfied in the context of first-order differential operators $\mS$ and $\D$ on smooth manifolds, when $\phi,v$ are compactly supported smooth functions and $\D_v$ is elliptic on $\supp(\phi)$ (this is the setting considered in \cite{vdD20_Kasp_open}). 
\end{remark}

\section{The Kasparov product of half-closed modules}
\label{sec:Kasp_prod}

\begin{assumption}
\label{ass:Kasp_prod}
Let $A$ be a ($\Z_2$-graded) separable $C^*$-algebra, let $B$ and $C$ be ($\Z_2$-graded) $\sigma$-unital $C^*$-algebras, and let $\A\subset A$ and $\B\subset B$ be dense $*$-subalgebras. 
Consider three half-closed modules $(\A,{}_{\pi_1}(E_1)_B,\D_1)$, $(\B,{}_{\pi_2}(E_2)_C,\D_2)$, and $(\A,{}_\pi E_C,\D)$, where $E := E_1\hot_B E_2$ and $\pi = \pi_1\hot1$, and suppose that $\pi_1$ is essential. 
We assume that the $*$-subalgebra $\A\subset A$ contains an (even) almost idempotent approximate unit $\{u_n\}_{n\in\N}$ for $A$. 
\end{assumption}
For ease of notation, we will usually identify $u_n\equiv u_n\hot1 \equiv \pi(u_n\hot1)$ on $E$. 
We recall the `partition of unity' $\{\chi_k^2\}_{k\in\N}$ from \cref{defn:partition}. 

\begin{prop}
\label{prop:conn+bdd-loc-pos}
In the setting of \cref{ass:Kasp_prod}, assume that the connection condition (\cref{defn:connection_condition}) is satisfied. 
For each $k\in\N$, consider elements $v_k,w_k\in\{u_n\}$ with $v_ku_{k+1} = w_ku_{k+1} = u_{k+1}$, and write $\D_k := v_k \D v_k$ and $\D_{1,k} := w_k \D_1 w_k$. 
We assume furthermore that for some $0<\kappa<2$ the following condition is satisfied:
\begin{itemize}
\item for each $k\in\N$ there exists $\alpha_k\in(0,\infty)$ such that $\chi_k [ F_{\D_k} , F_{\alpha_k\D_{1,k}}\hot1 ] \chi_k + \kappa \chi_k^2$ is positive modulo compact operators. 
\end{itemize}
Then $(\A,{}_\pi E_C,\D)$ represents the Kasparov product of $(\A,{}_{\pi_1}(E_1)_B,\D_1)$ and $(\B,{}_{\pi_2}(E_2)_C,\D_2)$. 
\end{prop}
\begin{proof}
We can represent $\D$ and $\D_2$ by their bounded transforms $F_\D$ and $F_{\D_2}$, and (by \cref{thm:local_rep_KK}) we can represent $\D_1$ by a localised representative $\til F_{\D_1}(\alpha)$. 
We have seen in \cref{prop:connection} that the connection condition of \cref{thm:Connes-Skandalis} is satisfied. 
Thus it remains to show that $F_\D$ and $\til F_{\D_1}(\alpha)$ (for the given sequence $\{\alpha_k\}_{k\in\N}$) satisfy the positivity condition of \cref{thm:Connes-Skandalis}. 

To prove the positivity condition, it suffices to consider $u_n [ F_\D , \til F_{\D_1}(\alpha)\hot1 ] u_n$, since we have norm-convergence $a u_n \to a$. 
From \cref{lem:almost_idempotent} we know that $\sum_k \chi_k u_n$ is a finite sum. 
We know from \cref{lem:local_comparison} (applied with $a=\chi_k$, $c=u_{k+1}$, and $b=v_k$) that $\chi_k (F_\D - F_{\D_k}) \sim 0$. 
Using furthermore that $[F_\D,\chi_k] \sim 0$ by \cref{thm:Hilsum}, we have 
\begin{align*}
u_n [ F_\D , \til F_{\D_1}(\alpha)\hot1 ] u_n  
&= \sum_k u_n [ F_\D , \chi_k (F_{\alpha_k\D_{1,k}}\hot1) \chi_k ] u_n \\ 
&\stackrel{\ref{thm:Hilsum}}{\sim} \sum_k u_n \chi_k [ F_\D , F_{\alpha_k\D_{1,k}}\hot1 ] \chi_k u_n  \\
&\stackrel{\ref{lem:local_comparison}}{\sim} \sum_k u_n \chi_k [ F_{\D_k} , F_{\alpha_k\D_{1,k}}\hot1 ] \chi_k u_n  .
\end{align*}
By hypothesis, we have a sequence $\{\alpha_k\}_{k\in\N} \subset (0,\infty)$ such that $\chi_k [ F_{\D_k} , F_{\alpha_k\D_{1,k}}\hot1 ] \chi_k + \kappa \chi_k^2$ is positive modulo compact operators. We then conclude that 
\begin{align*}
u_n [ F_D , \til F_{\D_1}(\alpha)\hot1 ] u_n  
&\sim \sum_k u_n \chi_k [ F_{\D_k} , F_{\alpha_k\D_{1,k}}\hot1 ] \chi_k u_n \\
&\gtrsim - \sum_k \kappa u_n \chi_k^2 u_n  
= - \kappa u_n^2 .
\end{align*}
Thus the positivity condition of \cref{thm:Connes-Skandalis} also holds, and the statement follows. 
\end{proof}

\begin{thm}
\label{thm:Kucerovsky_half-closed_localised}
In the setting of \cref{ass:Kasp_prod}, assume that the connection condition (\cref{defn:connection_condition}) is satisfied. 
We assume furthermore that for each $n\in\N$ the following conditions are satisfied:
\begin{itemize}
\item we have the domain inclusion $\Dom(\D u_n) \cap \Ran(u_n) \subset \Dom(\D_1u_n\hot1)$; 
\item there exists $c_n\in[0,\infty)$ such that for all $\psi\in\Dom(\D u_n) \cap \Ran(u_n)$ we have 
\begin{align}
\label{eq:positivity_localised}
\big\la u_n(\D_1\hot1)\psi \bigmvert \D u_n\psi \big\ra + \big\la \D u_n\psi \bigmvert u_n(\D_1\hot1)\psi \big\ra 
\geq - c_n \big\la \psi \bigmvert (1+(u_n\D u_n)^2)^{\frac12} \psi \big\ra . 
\end{align}
\end{itemize}
Then $(\A,{}_\pi E_C,\D)$ represents the Kasparov product of $(\A,{}_{\pi_1}(E_1)_B,\D_1)$ and $(\B,{}_{\pi_2}(E_2)_C,\D_2)$. 
\end{thm}
\begin{proof}
Let us write $v_k:=u_{k+2}$ and $w_k:=u_{k+3}$. 
For any $k\in\N$, we will check that \cref{ass:local_positivity} is satisfied by the operators $\D$, $\mS = \mS_k = \D_{1,k}\hot1 := w_k\D_1 w_k\hot1$, $\chi=\chi_k$, $\rho = u_{k+1}$, $\phi = \Id$, and $v=v_k$. We shall write $\D_k := \D_v = v_k \D v_k$. 

We have by assumption the domain inclusion $\Dom(\D v_k) \cap \Ran(v_k) \subset \Dom(\D_1v_k\hot1)$. Noting that $\Dom(\D_1v_k\hot1) = \Dom(\D_1w_k^2v_k\hot1) = \Dom(\mS_kv_k)$, we see that condition \ref{ass:loc_pos_Dom} is satisfied. 
Since $\phi=\Id$ and $\{u_n\}$ is almost idempotent, and using \cref{lem:almost_idempotent}, we see that condition \ref{ass:loc_pos_approx_unit} is satisfied. 
By assumption, we have $u_n\in\Lip^*(\D)\cap\Lip^*(\D_1\hot1)$, and since $\{u_n\}$ is commutative, it follows that we also have $u_n\in\Lip(\mS_k)$, so condition \ref{ass:loc_pos_Lip} is satisfied. 
Condition \ref{ass:loc_pos_cpt_res} follows from \cref{lem:cpt_res_localisation} (using that $\chi_k u_{k+1} = \chi_k$ and $u_{k+1} v_k = u_{k+1}$). 
Finally, for $\psi\in\Dom(\D v_k) \cap \Ran(v_k)$ we have $w_k\psi=\psi$, and then it follows from \cref{eq:positivity_localised} that 
\begin{align*}
Q_v(\psi) &= \big\la\D v_k\psi \bigmvert v_k\mS_k\psi \big\ra + \big\la v_k\mS_k\psi \bigmvert \D v_k\psi\big\ra \\
&= \big\la \D v_k\psi \bigmvert v_k(\D_1\hot1)\psi \big\ra + \big\la v_k(\D_1\hot1)\psi \bigmvert \D v_k\psi \big\ra \\
&\geq - c_{k+2} \big\la \psi \bigmvert (1+\D_k^2)^{\frac12}\psi \big\ra ,
\end{align*}
which shows condition \ref{ass:loc_pos_condition}. 
Thus \cref{ass:local_positivity} is indeed satisfied. 

Hence we can apply \cref{prop:local_positivity_phi=1} for each $k\in\N$, and for any $0<\kappa<2$ we obtain a sequence $\{\alpha_k\}_{k\in\N} \subset (0,\infty)$ such that $\chi_k [ F_{\D_k} , F_{\alpha_k\D_{1,k}}\hot1 ] \chi_k + \kappa \chi_k^2$ is positive modulo compact operators. 
The statement then follows from \cref{prop:conn+bdd-loc-pos}. 
\end{proof}

\begin{remark}
We note that \cref{eq:positivity_localised} may be replaced by 
\[
\big\la u_n(\D_1\hot1)u_n\psi \bigmvert u_n\D u_n\psi \big\ra + \big\la u_n\D u_n\psi \bigmvert u_n(\D_1\hot1)u_n\psi \big\ra 
\geq - c_n \big\la \psi \bigmvert (1+(u_n\D u_n)^2)^{\frac12} \psi \big\ra . 
\]
In this case the proof of \cref{thm:Kucerovsky_half-closed_localised} may be repeated, now choosing $w_k=u_{k+2}$ instead of $w_k=u_{k+3}$ (indeed, condition \ref{ass:loc_pos_condition} then follows from the above inequality, while for condition \ref{ass:loc_pos_Dom} we note that $\Dom(\D_1v_k\hot1) \subset \Dom(v_k\D_1v_k^2\hot1)$). 
\end{remark}

\begin{remark}
\label{remark:KS}
In the setting of \cref{thm:Kucerovsky_half-closed_localised}, suppose furthermore that $\D_1$ commutes with the approximate identity $u_n$. 
Then it follows from a standard argument that $\D_1$ is in fact self-adjoint. Indeed, if $[\D_1,u_n]=0$ then also $[\D_1^*,u_n]=0$. Using that $u_n\colon\Dom\D_1^*\to\Dom\D_1$, we then see for any $\xi\in\Dom\D_1^*$ that $\D_1 u_n\xi = u_n\D_1^*\xi$ converges to $\D_1^*\xi$, which proves that $\xi$ lies in the domain of $\D_1$. 
(For another argument for the self-adjointness of $\D_1$, see also \cite[Remark 6.5]{KS19}.) 

Furthermore, we note that $\{u_n\}_{n\in\N}$ is a \emph{localising subset} in the sense of \cite[Definition 6.2]{KS19}. 
The statement of \cref{thm:Kucerovsky_half-closed_localised} is then very similar to \cite[Theorem 6.10]{KS19}, which requires the following `local positivity condition' \cite[Definition 6.3]{KS19}:
\[
\big\la (\D_1\hot1)u_n\psi \bigmvert \D u_n\psi \big\ra + \big\la \D u_n\psi \bigmvert (\D_1\hot1)u_n\psi \big\ra \geq -c_n \la\psi|\psi\ra . 
\]
However, we note that our local positivity condition is a weaker assumption, since we allow for the additional factor $(1+(u_n\D u_n)^2)^{\frac12}$ in the inner product on the right-hand-side of the above inequality. 
Perhaps even more importantly, the major advantage of our proof of \cref{thm:Kucerovsky_half-closed_localised} is that we do \emph{not} require $u_n$ to commute with $\D_1$. This allows us to genuinely deal with the case of non-self-adjoint $\D_1$, whereas \cite{KS19} deals `essentially' only with the case where $\D_1$ is self-adjoint (as explained in \cite[Remark 6.5]{KS19}). 
\end{remark}

\subsection{Local Kucerovsky-type theorems}
\label{sec:local_Kucerovsky}

As already explained in the Introduction, we would like to replace the conditions in \cref{thm:Kucerovsky_half-closed_localised} by the following more natural condition (and in \S\ref{sec:strong_local_positivity} and \S\ref{sec:C2_local_positivity} we will consider two possible sufficient conditions which allow us to make this additional step). 

\begin{defn}
\label{defn:local_positivity_condition}
In the setting of \cref{ass:Kasp_prod}, 
the \emph{local positivity condition} requires that for each $n\in\N$ the following assumptions hold: 
\begin{enumerate}
\item we have the inclusion $u_n\cdot\Dom(\D) \subset \Dom(\D_1\hot1)$; 
\item there exists $c_n\in[0,\infty)$ such that for all $\psi\in\Dom(\D)$ we have 
\[
\big\la (\D_1\hot1)u_n\psi \bigmvert \D u_n\psi \big\ra + \big\la \D u_n\psi \bigmvert (\D_1\hot1)u_n\psi \big\ra \geq -c_n \big\la u_n\psi \bigmvert (1+\D^*\D)^{1/2} u_n\psi \big\ra . 
\] 
\end{enumerate}
\end{defn}

\begin{remark}
For the domain condition (1) in \cref{defn:local_positivity_condition}, it is of course sufficient (but not necessary!) to have the domain inclusion $\Dom(\D)\subset\Dom(\D_1\hot1)$. 
\end{remark}

\begin{lem}
\label{lem:localising_positivity}
Consider $\D_k := u_{k+4}\D u_{k+4}$ and $\mS_k := u_{k+4}\D_1u_{k+4}\hot1$, and $\psi\in\Dom\D$. Then we have 
\begin{align*}
&2\Re \big\la \D_k u_{k+2} \psi \bigmvert \mS_k u_{k+2} \psi \big\ra \\
&\quad= 2\Re \big\la \D u_{k+2} \psi \bigmvert (\D_1\hot1) u_{k+2} \psi \big\ra + 2\Re \big\la [u_{k+4}^2,\D] u_{k+2} \psi \bigmvert [\D_1\hot1,u_{k+3}] u_{k+2} \psi \big\ra \\
&\quad\geq 2\Re \big\la \D u_{k+2} \psi \bigmvert (\D_1\hot1) u_{k+2} \psi \big\ra - 2 \big\| [D,u_{k+4}^2] [\D_1\hot1,u_{k+3}] \big\| \, \big\la u_{k+2} \psi \bigmvert u_{k+2} \psi \big\ra . 
\end{align*}
\end{lem}
\begin{proof}
We compute
\begin{align*}
&\big\la \D_k u_{k+2} \psi \bigmvert \mS_k u_{k+2} \psi \big\ra 
= \big\la u_{k+4} \D u_{k+2} \psi \bigmvert u_{k+4} (\D_1\hot1) u_{k+3} u_{k+2} \psi \big\ra \\ 
&\quad= \big\la u_{k+4} \D u_{k+2} \psi \bigmvert u_{k+4} [\D_1\hot1,u_{k+3}] u_{k+2} \psi \big\ra + \big\la \D u_{k+2} \psi \bigmvert u_{k+3} (\D_1\hot1) u_{k+2} \psi \big\ra \\ 
&\quad= \big\la [u_{k+4}^2,\D] u_{k+2} \psi \bigmvert [\D_1\hot1,u_{k+3}] u_{k+2} \psi \big\ra + \big\la \D u_{k+2} \psi \bigmvert [\D_1\hot1,u_{k+3}] u_{k+2} \psi \big\ra \\
&\qquad+ \big\la \D u_{k+2} \psi \bigmvert u_{k+3} (\D_1\hot1) u_{k+2} \psi \big\ra \\ 
&\quad= \big\la \D u_{k+2} \psi \bigmvert (\D_1\hot1) u_{k+2} \psi \big\ra + \big\la [u_{k+4}^2,\D] u_{k+2} \psi \bigmvert [\D_1\hot1,u_{k+3}] u_{k+2} \psi \big\ra . 
\end{align*}
The inequality is then clear. 
\end{proof}

\subsubsection{Strong local positivity}
\label{sec:strong_local_positivity}

\begin{defn}
\label{defn:strong_positivity_condition}
In the setting of \cref{ass:Kasp_prod}, 
the \emph{strong local positivity condition} requires that for each $n\in\N$ the following assumptions hold: 
\begin{enumerate}
\item we have the inclusion $u_n\cdot\Dom(\D) \subset \Dom(\D_1\hot1)$; 
\item there exist $\nu_n\in(0,\infty)$ and $c_n\in[0,\infty)$ such that for all $\psi\in\Dom(\D)$ we have 
\begin{multline*}
\big\la (\D_1\hot1)u_n\psi \bigmvert \D u_n\psi \big\ra + \big\la \D u_n\psi \bigmvert (\D_1\hot1)u_n\psi \big\ra \\*
\geq \nu_n \big\lla (\D_1\hot1)u_n\psi \big\rra - c_n \big\la u_n\psi \bigmvert (1+\D^*\D)^{\frac12} u_n\psi \big\ra . 
\end{multline*}
\end{enumerate}
\end{defn}

\begin{lem}
\label{lem:strong_localised_positivity}
In the setting of \cref{ass:Kasp_prod}, assume that the strong local positivity condition is satisfied. 
Then writing $\mS_k := u_{k+4}\D_1 u_{k+4}\hot1$, $\D_k := u_{k+4}\D u_{k+4}$, and $\phi_k := u_{k+2}$, 
there exists for each $k\in\N$ a constant $d_k\in[0,\infty)$ such that for all $\psi\in\Dom(\D)$ we have 
\[
\big\la \mS_k \phi_k\psi \bigmvert \D_k \phi_k\psi \big\ra + \big\la \D_k \phi_k\psi \bigmvert \mS_k \phi_k\psi \big\ra 
\geq \nu_{k+2} \big\lla u_{k+4} \mS_k \phi_k\psi \big\rra - d_k \big\la \phi_k\psi \bigmvert (1+\D_k^2)^{1/2} \phi_k\psi \big\ra . 
\] 
\end{lem}
\begin{proof}
Combining \cref{lem:localising_positivity} with the strong local positivity condition, we have 
\begin{align}
\label{eq:localised_pos_strong}
2\Re \big\la \D_k \phi_k \psi \bigmvert \mS_k \phi_k \psi \big\ra 
&\geq 2\Re \big\la \D \phi_k \psi \bigmvert (\D_1\hot1) \phi_k \psi \big\ra - 2 \big\| [D,u_{k+4}^2] [\D_1\hot1,u_{k+3}] \big\| \, \big\lla \phi_k \psi \big\rra \nonumber\\ 
&\geq \nu_{k+2} \big\lla (\D_1\hot1)\phi_k\psi \big\rra 
- c_{k+2} \big\la \phi_k\psi \bigmvert (1+\D^*\D)^{\frac12} \phi_k\psi \big\ra \nonumber\\
&\quad- 2 \big\| [D,u_{k+4}^2] [\D_1\hot1,u_{k+3}] \big\| \, \big\lla \phi_k \psi \big\rra . 
\end{align}
The first term in the last expression can be estimated by 
\begin{align*}
\big\lla \D_1\hot1 \phi_k\psi \big\rra 
&= \big\la (\D_1\hot1) u_{k+4}^4 \phi_k\psi \bigmvert (\D_1\hot1) \phi_k\psi \big\ra \\
&= \big\la u_{k+4}^2 (\D_1\hot1) \phi_k\psi \bigmvert u_{k+4}^2 (\D_1\hot1) \phi_k\psi \big\ra \\
&\qquad+ \big\la [\D_1\hot1,u_{k+4}^4] \phi_k\psi \bigmvert (\D_1\hot1) u_{k+3} \phi_k\psi \big\ra \\
&= \big\lla u_{k+4} \mS_k \phi_k\psi \big\rra 
+ \big\la [\D_1\hot1,u_{k+4}^4] \phi_k\psi \bigmvert [\D_1\hot1,u_{k+3}] \phi_k\psi \big\rra \\
&\qquad+ \big\la [\D_1\hot1,u_{k+4}^4] \phi_k\psi \bigmvert u_{k+3} (\D_1\hot1) \phi_k\psi \big\rra \\
&\geq \big\lla u_{k+4} \mS_k \phi_k\psi \big\rra 
- \big\| [\D_1\hot1,u_{k+4}^4] [\D_1\hot1,u_{k+3}] \big\| \, \big\lla \phi_k\psi \big\rra ,
\end{align*}
where we used that $u_{k+3} [\D_1\hot1,u_{k+4}^4] \phi_k = 0$. 
The second term in \cref{eq:localised_pos_strong} can be estimated as follows. 
Since $u_{k+4}^2\colon\Dom(u_{k+4}\D u_{k+4})\to\Dom\D$, we know that the symmetric operator $T := u_{k+4}^2 (1+\D^*\D)^{\frac12} u_{k+4}^2$ is relatively bounded by $\D_k$. Using \cref{lem:interpolation}, we find that 
\begin{align}
\label{eq:bound_D-D_k}
\big\la \phi_k\psi \bigmvert (1+\D^*\D)^{1/2} \phi_k\psi \big\ra &= \big\la \phi_k\psi \bigmvert u_{k+4}^2 (1+\D^*\D)^{1/2} u_{k+4}^2 \phi_k\psi \big\ra \nonumber\\
&\leq \big\| T (1+\D_k^2)^{-\frac12} \big\| \; \big\la \phi_k\psi \bigmvert (1+\D_k^2)^{1/2} \phi_k\psi \big\ra . 
\end{align}
We then obtain the desired inequality with 
\begin{align*}
d_k &= \nu_{k+2} \big\| [\D_1\hot1,u_{k+4}^4] [\D_1\hot1,u_{k+3}] \big\| + c_{k+2} \big\| T (1+\D_k^2)^{-\frac12} \big\| \\
&\quad+ 2 \big\| [D,u_{k+4}^2] [\D_1\hot1,u_{k+3}] \big\| . 
\qedhere
\end{align*}
\end{proof}

\begin{thm}
\label{thm:Kucerovsky_half-closed_strong}
%
We consider the setting of \cref{ass:Kasp_prod}: let $A$ be a ($\Z_2$-graded) separable $C^*$-algebra, let $B$ and $C$ be ($\Z_2$-graded) $\sigma$-unital $C^*$-algebras, and let $\A\subset A$ and $\B\subset B$ be dense $*$-subalgebras. 
Consider three half-closed modules $(\A,{}_{\pi_1}(E_1)_B,\D_1)$, $(\B,{}_{\pi_2}(E_2)_C,\D_2)$, and $(\A,{}_\pi E_C,\D)$, where $E := E_1\hot_B E_2$ and $\pi = \pi_1\hot1$, and suppose that $\pi_1$ is essential. 
We assume that the $*$-subalgebra $\A\subset A$ contains an (even) almost idempotent approximate unit $\{u_n\}_{n\in\N}$ for $A$. 
We assume furthermore that the following conditions are satisfied: 
\begin{description}
\item[Connection condition \textnormal{(\cref{defn:connection_condition})}:] 
for all $\psi$ in a dense subspace $\E_1$ of $\Dom\D_1$, we have 
\[
\til T_\psi := \mattwo{0}{T_\psi}{T_\psi^*}{0} \in \Lip(\D\oplus\D_2) ; 
\]
\item[Strong local positivity condition \textnormal{(\cref{defn:strong_positivity_condition})}:] 
for each $n\in\N$ the following assumptions hold: 
\begin{enumerate}
\item we have the inclusion $u_n\cdot\Dom(\D) \subset \Dom(\D_1\hot1)$; 
\item there exist $\nu_n\in(0,\infty)$ and $c_n\in[0,\infty)$ such that for all $\psi\in\Dom(\D)$ we have 
\begin{multline*}
\big\la (\D_1\hot1)u_n\psi \bigmvert \D u_n\psi \big\ra + \big\la \D u_n\psi \bigmvert (\D_1\hot1)u_n\psi \big\ra \\*
\geq \nu_n \big\lla (\D_1\hot1)u_n\psi \big\rra - c_n \big\la u_n\psi \bigmvert (1+\D^*\D)^{\frac12} u_n\psi \big\ra . 
\end{multline*}
\end{enumerate}
\end{description}
Then $(\A,{}_\pi E_C,\D)$ represents the Kasparov product of $(\A,{}_{\pi_1}(E_1)_B,\D_1)$ and $(\B,{}_{\pi_2}(E_2)_C,\D_2)$. 
\end{thm}
\begin{proof}
Let us write $v_k:=w_k:=u_{k+4}$. 
For any $k\in\N$, we will check that \cref{ass:local_positivity} is satisfied by the operators $\D$, $\mS = \mS_k = \D_{1,k}\hot1 := w_k\D_1 w_k\hot1$, $\chi=\chi_k$, $\rho = u_{k+1}$, $\phi = u_{k+2}$, and $v=v_k$. We shall write $\D_k := \D_v = v_k \D v_k$. 

We have by assumption the domain inclusion $u_n\cdot\Dom(\D) \subset \Dom(\D_1\hot1)$ for any $n\in\N$. For any $\psi\in\Dom(\D v)$ we note that $v\psi = u_{k+5} v\psi \in u_{k+5}\cdot\Dom(\D) \subset \Dom(\D_1\hot1)$, and therefore $\psi\in\Dom(\mS_k v)$, which shows condition \ref{ass:loc_pos_Dom}. 
Conditions \ref{ass:loc_pos_approx_unit}-\ref{ass:loc_pos_cpt_res} follow as in the proof of \cref{thm:Kucerovsky_half-closed_localised}. 
Furthermore, for any $\psi\in\Dom(\D v)$ we have $\phi\psi=\phi v\psi$, where $v\psi\in\Dom\D$. 
Condition \ref{ass:strong_loc_pos_condition} is then satisfied by \cref{lem:strong_localised_positivity}. 
So by \cref{prop:local_positivity_strong} we know that there exists a sequence $\{\alpha_k\}_{k\in\N} \subset (0,\infty)$ such that $\chi_k [ F_{\D_k} , F_{\alpha_k\D_{1,k}}\hot1 ] \chi_k + \kappa \chi_k^2$ is positive modulo compact operators. 
The statement then follows from \cref{prop:conn+bdd-loc-pos}. 
\end{proof}

\subsubsection{A `differentiability' condition}
\label{sec:C2_local_positivity}

\begin{lem}
\label{lem:localised_positivity_C2}
In the setting of \cref{ass:Kasp_prod}, assume that the local positivity condition is satisfied. 
Then writing $\mS_k := u_{k+4}\D_1 u_{k+4}\hot1$, $\D_k := u_{k+4}\D u_{k+4}$, and $\phi_k := u_{k+2}$, 
there exists for each $k\in\N$ a constant $d_k\in[0,\infty)$ such that for all $\psi\in\Dom(\D)$ we have 
\[
\big\la \mS_k \phi_k\psi \bigmvert \D_k \phi_k\psi \big\ra + \big\la \D_k \phi_k\psi \bigmvert \mS_k \phi_k\psi \big\ra \geq -d_k \big\la \phi_k\psi \bigmvert (1+\D_k^2)^{1/2} \phi_k\psi \big\ra . 
\] 
\end{lem}
\begin{proof}
Combining \cref{lem:localising_positivity} with the local positivity condition, we have 
\begin{align*}
&2\Re \big\la \D_k \phi_k \psi \bigmvert \mS_k \phi_k \psi \big\ra \\
&\quad\geq 2\Re \big\la \D \phi_k \psi \bigmvert (\D_1\hot1) \phi_k \psi \big\ra - 2 \big\| [D,u_{k+4}^2] [\D_1\hot1,u_{k+3}] \big\| \, \big\la \phi_k \psi \bigmvert \phi_k \psi \big\ra \\ 
&\quad\geq - c_{k+2} \big\la \phi_k\psi \bigmvert (1+\D^*\D)^{\frac12} \phi_k\psi \big\ra 
- 2 \big\| [D,u_{k+4}^2] [\D_1\hot1,u_{k+3}] \big\| \, \big\la \phi_k \psi \bigmvert \phi_k \psi \big\ra . 
\end{align*}
As in \cref{eq:bound_D-D_k}, an application of \cref{lem:integral_estimate_first-order2} with $T := u_{k+4}^2 (1+\D^*\D)^{\frac12} u_{k+4}^2$ yields 
\begin{align*}
\big\la \phi_k\psi \bigmvert (1+\D^*\D)^{1/2} \phi_k\psi \big\ra 
&\leq \big\| T (1+\D_k^2)^{-\frac12} \big\| \; \big\la \phi_k\psi \bigmvert (1+\D_k^2)^{1/2} \phi_k\psi \big\ra . 
\end{align*}
Thus we obtain the desired inequality with 
\[
d_k = c_{k+2} \big\| T (1+\D_k^2)^{-\frac12} \big\| + 2 \big\| [D,u_{k+4}^2] [\D_1\hot1,u_{k+3}] \big\| . 
\qedhere
\]
\end{proof}

\begin{thm}
\label{thm:Kucerovsky_half-closed_C2}
%
We consider the setting of \cref{ass:Kasp_prod}: let $A$ be a ($\Z_2$-graded) separable $C^*$-algebra, let $B$ and $C$ be ($\Z_2$-graded) $\sigma$-unital $C^*$-algebras, and let $\A\subset A$ and $\B\subset B$ be dense $*$-subalgebras. 
Consider three half-closed modules $(\A,{}_{\pi_1}(E_1)_B,\D_1)$, $(\B,{}_{\pi_2}(E_2)_C,\D_2)$, and $(\A,{}_\pi E_C,\D)$, where $E := E_1\hot_B E_2$ and $\pi = \pi_1\hot1$, and suppose that $\pi_1$ is essential. 
We assume that the $*$-subalgebra $\A\subset A$ contains an (even) almost idempotent approximate unit $\{u_n\}_{n\in\N}$ for $A$. 
We assume furthermore that the following conditions are satisfied: 
\begin{description}
\item[Connection condition \textnormal{(\cref{defn:connection_condition})}:] 
for all $\psi$ in a dense subspace $\E_1$ of $\Dom\D_1$, we have 
\[
\til T_\psi := \mattwo{0}{T_\psi}{T_\psi^*}{0} \in \Lip(\D\oplus\D_2) ; 
\]
\item[Local positivity condition \textnormal{(\cref{defn:local_positivity_condition})}:] 
for each $n\in\N$ the following assumptions hold: 
\begin{enumerate}
\item we have the inclusion $u_n\cdot\Dom(\D) \subset \Dom(\D_1\hot1)$; 
\item there exists $c_n\in[0,\infty)$ such that for all $\psi\in\Dom(\D)$ we have 
\[
\big\la (\D_1\hot1)u_n\psi \bigmvert \D u_n\psi \big\ra + \big\la \D u_n\psi \bigmvert (\D_1\hot1)u_n\psi \big\ra \geq -c_n \big\la u_n\psi \bigmvert (1+\D^*\D)^{1/2} u_n\psi \big\ra . 
\] 
\end{enumerate}
\end{description}
Finally, assume that the following additional condition holds:
\begin{enumerate}[label=\textnormal{($\star$)}]
\item \label{item:differentiability}
for each $n\in\N$, $u_n[\D,u_n]u_{n+2}$ maps $\Dom(\D u_{n+2}^2)$ to $\Dom(\D_1\hot1)$. 
\end{enumerate}
Then $(\A,{}_\pi E_C,\D)$ represents the Kasparov product of $(\A,{}_{\pi_1}(E_1)_B,\D_1)$ and $(\B,{}_{\pi_2}(E_2)_C,\D_2)$. 
\end{thm}
\begin{proof}
Most of the proof is similar to the proof of \cref{thm:Kucerovsky_half-closed_strong}, so we shall be brief here. 
Conditions \ref{ass:loc_pos_Dom}-\ref{ass:loc_pos_cpt_res} in \cref{ass:local_positivity} are again satisfied by the operators $\D$, $\mS = \D_{1,k}\hot1 := w_k\D_1 w_k\hot1$, $\chi=\chi_k$, $\rho = u_{k+1}$, $\phi = u_{k+2}$, and $v=v_k$, where $v_k=w_k=u_{k+4}$. 
Condition \ref{ass:loc_pos_condition} is shown in \cref{lem:localised_positivity_C2}. Thus \cref{ass:local_positivity} is indeed satisfied. 
Furthermore, since $\phi_k [\D_k,\phi_k] = u_{k+2} [\D,u_{k+2}] u_{k+4}$ maps $\Dom(\D_k)$ to $\Dom(\D_1\hot1)\subset\Dom(\mS_k)$ by condition \ref{item:differentiability}, also condition \ref{ass:weak_C2} is satisfied. 
Hence we can apply \cref{prop:local_positivity_C2} for each $k\in\N$, and we obtain a sequence $\{\alpha_k\}_{k\in\N} \subset (0,\infty)$ such that $\chi_k [ F_{\D_k} , F_{\alpha_k\D_{1,k}}\hot1 ] \chi_k + \kappa \chi_k^2$ is positive modulo compact operators. 
The statement then follows from \cref{prop:conn+bdd-loc-pos}. 
\end{proof}

\begin{remark}
If for each $n\in\N$, $u_n[\D,u_n]$ maps $\Dom(\D)$ to $\Dom(\D_1\hot1)$, and $u_{n+2}$ commutes with $[\D,u_n]$, then one can check that the condition \ref{item:differentiability} is satisfied. 
This situation for instance occurs if $\D$ and $\D_1$ are first-order differential operators on smooth manifolds (with $\D$ elliptic) and $u_n$ are compactly supported smooth functions. This is precisely the setting studied in \cite{vdD20_Kasp_open}, and \cref{thm:Kucerovsky_half-closed_C2} can be viewed as a generalisation of the statement and proof of \cite[Theorem 4.2]{vdD20_Kasp_open}. 
\end{remark}

\subsection{The constructive approach}
\label{sec:construction}

In this subsection we show that the strong local positivity condition is quite naturally satisfied in (a localised version of) the constructive approach to the unbounded Kasparov product. 
In the following, we will in particular assume the existence of a suitable operator $\mT$. 
As explained in \cref{remark:construct_T}, one should keep in mind here the case where (under suitable assumptions) $\mT = 1\hot_\nabla\D_2$ is constructed from a suitable `connection' $\nabla$ on $E_1$. 

\begin{assumption}
\label{ass:construction}
Let $A$ be a ($\Z_2$-graded) separable $C^*$-algebra, let $B$ and $C$ be ($\Z_2$-graded) $\sigma$-unital $C^*$-algebras, and let $\A\subset A$ and $\B\subset B$ be dense $*$-subalgebras. 
Consider two half-closed modules $(\A,{}_{\pi_1}(E_1)_B,\D_1)$ and $(\B,{}_{\pi_2}(E_2)_C,\D_2)$, and suppose that $\pi_1$ is essential. 
We write $E := E_1\hot_B E_2$, $\pi = \pi_1\hot1$, and $\mS := \D_1\hot1$, and we consider an odd symmetric operator $\mT$ on $E$. 
We denote by $\D := \bar{\mS+\mT}$ the closure of the operator $\mS+\mT$ on the initial domain $\Dom(\mS)\cap\Dom(\mT)$. 
We assume that the following conditions are satisfied:
\begin{enumerate}[label={(A\arabic*)}]
\item \label{item:S+T_half-closed}
the intersection $\Dom(\mS)\cap\Dom(\mT)$ is dense in $E$, and 
the triple $(\A,{}_\pi E_C,\D)$ is a half-closed module;
\item \label{item:T_connection}
for all $\psi$ in a dense subset of $\A\cdot\Dom\D_1$, we have 
\[
\til T_\psi := \mattwo{0}{T_\psi}{T_\psi^*}{0} \in \Lip(\mT\oplus\D_2) ; 
\]
\item \label{item:loc_dom_incl}
the $*$-subalgebra $\A\subset A$ contains an almost idempotent approximate unit $\{u_n\}_{n\in\N}$ for $A$, such that we have the inclusion 
\[
u_n \cdot \Dom(\D) \subset \Dom(\mS) \cap \Dom(\mT) . 
\]
\end{enumerate}
\end{assumption}

\begin{prop}
\label{prop:Kasp_prod_half-closed}
Consider the setting of \cref{ass:construction}. 
We assume furthermore that for each $n\in\N$ there exists $c_n\in[0,\infty)$ such that for all $\psi\in\Dom\D$ we have the inequality
\begin{align}
\label{eq:anti-comm_ineq}
\pm \Big( \big\la \mS u_n\psi \bigmvert \mT u_n\psi \big\ra + \big\la \mT u_n\psi \bigmvert \mS u_n\psi \big\ra \Big) 
&\leq c_n \big\la u_n\psi \bigmvert (1+\D^*\D)^{\frac12} u_n\psi \big\ra . 
\end{align}
Then $(\A,{}_\pi E_C,\D)$ represents the Kasparov product of $(\A,{}_{\pi_1}(E_1)_B,\D_1)$ and $(\B,{}_{\pi_2}(E_2)_C,\D_2)$. 
\end{prop}
\begin{proof}
For any $\psi\in\Dom\D_1$ we have bounded operators $\mS T_\psi = T_{\D_1\psi}$ and $T^*_\psi \mS = T^*_{\D_1\psi}$, which means that $\til T_\psi\in\Lip(\mS\oplus0)$. Combined with condition \ref{item:T_connection} this ensures that $\til T_\psi$ preserves $\big(\Dom(\mS)\cap\Dom(\mT)\big)\oplus\Dom(\D_2)$, and that $[\D\oplus\D_2,\til T_\psi]$ is bounded on this domain. Since $\Dom(\mS)\cap\Dom(\mT)$ is a core for $\D$, it follows that the connection condition (\cref{defn:connection_condition}) is satisfied. 

Next, using \cref{eq:anti-comm_ineq} we obtain the inequality 
\begin{align*}
\big\la \mS u_n\psi \bigmvert \D u_n\psi \big\ra + \big\la \D u_n\psi \bigmvert \mS u_n\psi \big\ra 
&= 2 \big\lla \mS u_n\psi \big\rra + \big\la \mS u_n\psi \bigmvert \mT u_n\psi \big\ra + \big\la \mT u_n\psi \bigmvert \mS u_n\psi \big\ra \\
&\geq 2 \big\lla \mS u_n\psi \big\rra - c_n \big\la u_n\psi \bigmvert (1+\D^*\D)^{\frac12} u_n\psi \big\ra . 
\end{align*}
We conclude that the strong local positivity condition (\cref{defn:strong_positivity_condition}) is also satisfied (with $\nu_n=2$ for all $n\in\N$). 
The statement then follows from \cref{thm:Kucerovsky_half-closed_strong}. 
\end{proof}

The following result provides a sufficient condition which ensures that the hypothesis of the above proposition is satisfied, by considering the anti-commutator $[\mS,\mT]$ on a suitable domain. 

\begin{thm}
\label{thm:Kasp_prod_half-closed}
Consider the setting of \cref{ass:construction}. 
We assume furthermore that there exists a core $\mF\subset\Dom\D$ 
such that 
for each $n\in\N$ we have $u_n\cdot\mF \subset \Dom(\mS\mT)\cap\Dom(\mT\mS)$, and there exists $C_n\in[0,\infty)$ such that for all $\eta\in\mF$ we have 
\begin{align}
\label{eq:anti-comm_bound}
\|[\mS,\mT]u_n\eta\| \leq C_n \|u_n\eta\|_\D . 
\end{align}
Then $(\A,{}_\pi E_C,\D)$ represents the Kasparov product of $(\A,{}_{\pi_1}(E_1)_B,\D_1)$ and $(\B,{}_{\pi_2}(E_2)_C,\D_2)$. 
\end{thm}
\begin{proof}
We consider the closed symmetric operator $\bar{[\mS,\mT]}$. 
Fix $n\in\N$, and let $\psi\in\Dom(\D)$. 
Since $\mF$ is a core for $\D$, we can choose a sequence $\{\psi_k\}_{k\in\N} \subset \mF$ such that $\|\psi_k-\psi\|_\D\to0$ as $k\to\infty$. The inequality 
\[
\big\| u_n \psi_k - u_n\psi \big\|_\D^2 
\leq 2 \big( \|u_n\|^2 + \|[\D,u_n]\|^2 \big) \big\| \psi_k - \psi \big\|_\D^2 
\]
ensures that we also have the convergence $\|u_n\psi_k-u_n\psi\|_\D\to0$ as $k\to\infty$. 
In particular, $u_n\psi_k$ is Cauchy with respect to $\|\cdot\|_\D$, and from \cref{eq:anti-comm_bound} we see that $u_n\psi_k$ is also Cauchy with respect to the graph norm of $[\mS,\mT]$. Hence $u_n\psi = \lim_{k\to\infty} u_n\psi_k$ lies in the domain of the closure $\bar{[\mS,\mT]}$, and we have shown the inclusion $u_n\cdot\Dom(\D) \subset \Dom(\bar{[\mS,\mT]})$. 

Furthermore, since $u_n\psi \in \Dom(\D)$, we can also choose a sequence $\{\eta_k\}_{k\in\N} \subset \mF$ such that $\|\eta_k-u_n\psi\|_\D\to0$ as $k\to\infty$. The inequality 
\[
\big\| u_{n+1} \eta_k - u_n\psi \big\|_\D^2 
= \big\| u_{n+1} (\eta_k - u_n\psi) \big\|_\D^2 
\leq 2 \big( \|u_{n+1}\|^2 + \|[\D,u_{n+1}]\|^2 \big) \big\| \eta_k - u_n\psi \big\|_\D^2 
\]
ensures that we then have the convergence $\|u_{n+1}\eta_k-u_n\psi\|_\D\to0$ as $k\to\infty$. 
Since $u_{n+1}\cdot\Dom(\D) \subset \Dom(\mS)$ by condition \ref{item:loc_dom_incl}, we can estimate 
\[
\big\| \mS u_{n+1} \eta_k - \mS u_n\psi \big\| \leq \big\| \mS u_{n+1} (1+\D^*\D)^{-\frac12} \big\| \; \big\| \eta_k-u_n\psi \big\|_\D .
\]
This ensures that we have norm-convergence $\mS u_{n+1}\eta_k \to \mS u_n\psi$ as $k\to\infty$. 
Similarly we also have norm-convergence $\mT u_{n+1}\eta_k \to \mT u_n\psi$ and $[\mS,\mT]u_{n+1}\eta_k \to \bar{[\mS,\mT]}u_n\psi$. 
Hence we have 
\begin{align*}
\big\la \mS u_n\psi \bigmvert \mT u_n\psi \big\ra + \big\la \mT u_n\psi \bigmvert \mS u_n\psi \big\ra
&= \lim_{k\to\infty} \big\la \mS u_{n+1}\eta_k \bigmvert \mT u_{n+1}\eta_k \big\ra + \big\la \mT u_{n+1}\eta_k \bigmvert \mS u_{n+1}\eta_k \big\ra \\
&= \lim_{k\to\infty} \big\la u_{n+1}\eta_k \bigmvert [\mS,\mT] u_{n+1}\eta_k \big\ra \\
&= \big\la u_n\psi \bigmvert \bar{[\mS,\mT]} u_n\psi \big\ra . 
\end{align*}
Finally, we can estimate the right-hand-side as follows. 
Since $\bar{[\mS,\mT]} u_{n+1} (1+\D^*\D)^{-\frac12}$ is bounded, and writing $c_n := \big\| \bar{[\mS,\mT]} u_{n+1} (1+\D^*\D)^{-\frac12} \big\|$, we know from \cref{lem:interpolation} that $\big\| (1+\D^*\D)^{-\frac14} u_{n+1} \bar{[\mS,\mT]} u_{n+1} (1+\D^*\D)^{-\frac14} \big\| \leq c_n$. 
For any $\psi\in\Dom(\D)$ we then have the inequality 
\begin{align*}
&\pm \big\la u_n\psi \bigmvert \bar{[\mS,\mT]} u_n\psi \big\ra \nonumber\\
&\quad= \pm \big\la (1+\D^*\D)^{\frac14} u_n \psi \bigmvert (1+\D^*\D)^{-\frac14} u_{n+1} \bar{[\mS,\mT]} u_{n+1} (1+\D^*\D)^{-\frac14} \, (1+\D^*\D)^{\frac14} u_n \psi \big\ra \nonumber\\
&\quad\leq c_n \big\la u_n\psi \bigmvert (1+\D^*\D)^{\frac12} u_n\psi \big\ra . 
\end{align*}
The statement then follows from \cref{prop:Kasp_prod_half-closed}. 
\end{proof}


\begin{thebibliography}{BDT89}

\bibitem[BDT89]{BDT89}
P.~Baum, R.~G. Douglas, and M.~E. Taylor, \doilinktitle{\emph{Cycles and
  relative cycles in analytic {$K$}-homology}}{10.4310/jdg/1214443829},
  \doilinkjournal{J. Diff. Geom.}{10.4310/jdg/1214443829} \doilinkvynp{
  \textbf{30} (1989), no.~3, 761--804}{10.4310/jdg/1214443829}.

\bibitem[BJ83]{BJ83}
S.~{Baaj} and P.~{Julg}, \emph{{Th\'eorie bivariante de {K}asparov et
  op\'erateurs non born\'es dans les {$C^{\ast} $}-modules hilbertiens}}, C. R.
  Acad. Sci. Paris S\'er. I Math. \textbf{296} (1983), 875--878.

\bibitem[Bla98]{Blackadar98}
B.~Blackadar, \emph{K-theory for operator algebras}, 2nd ed., Math. Sci. Res.
  Inst. Publ., Cambridge University Press, 1998.

\bibitem[Bla06]{Blackadar06}
\bysame, \emph{Operator algebras: Theory of {$C^{\ast}$}-algebras and von
  {N}eumann algebras}, Encyclopaedia of Mathematical Sciences, vol. 122,
  Springer, 2006.

\bibitem[BMS16]{BMS16}
S.~{Brain}, B.~{Mesland}, and W.~D. van Suijlekom, \doilinktitle{\emph{Gauge
  theory for spectral triples and the unbounded {K}asparov
  product}}{10.4171/JNCG/230}, \doilinkjournal{J. Noncommut.
  Geom.}{10.4171/JNCG/230} \doilinkvynp{ \textbf{10} (2016),
  135--206}{10.4171/JNCG/230}.

\bibitem[CS84]{CS84}
A.~Connes and G.~Skandalis, \doilinktitle{\emph{The longitudinal index theorem
  for foliations}}{10.2977/prims/1195180375}, \doilinkjournal{Publ. Res. Inst.
  Math. Sci.}{10.2977/prims/1195180375} \doilinkvynp{ \textbf{20} (1984),
  no.~6, 1139--1183}{10.2977/prims/1195180375}.

\bibitem[Dun20]{vdD20_Kasp_open}
K.~van~den Dungen, \doilinktitle{\emph{{The Kasparov product on submersions of
  open manifolds}}}{10.1142/S1793525320500454}, \doilinkjournal{J. Topol.
  Anal.}{10.1142/S1793525320500454} \doilinkvynp{
  (2020)}{10.1142/S1793525320500454}, published online.

\bibitem[Hig89]{Hig89pre}
N.~Higson, \emph{{K-homology and operators on non-compact manifolds}}, 1989,
  \href{http://www.personal.psu.edu/ndh2/math/Unpublished.html}{unpublished
  preprint}.

\bibitem[Hil10]{Hil10}
M.~Hilsum, \doilinktitle{\emph{Bordism invariance in
  {KK}-theory}}{10.7146/math.scand.a-15143}, \doilinkjournal{Math.
  Scand.}{10.7146/math.scand.a-15143} \doilinkvynp{ \textbf{107} (2010), no.~1,
  73--89}{10.7146/math.scand.a-15143}.

\bibitem[HR00]{Higson-Roe00}
N.~Higson and J.~Roe, \emph{{Analytic K-Homology}}, Oxford University Press,
  New York, 2000.

\bibitem[Kas80]{Kas80b}
G.~G. Kasparov, \doilinktitle{\emph{{The operator K-functor and extensions of
  {$C^{\ast}$}-algebras}}}{10.1070/IM1981v016n03ABEH001320},
  \doilinkjournal{Izv. Akad. Nauk SSSR}{10.1070/IM1981v016n03ABEH001320}
  \doilinkvynp{ \textbf{44} (1980), 571--636}{10.1070/IM1981v016n03ABEH001320}.

\bibitem[KL13]{KL13}
J.~Kaad and M.~Lesch, \doilinktitle{\emph{Spectral flow and the unbounded
  {K}asparov product}}{10.1016/j.aim.2013.08.015}, \doilinkjournal{Adv.
  Math.}{10.1016/j.aim.2013.08.015} \doilinkvynp{ \textbf{248} (2013),
  495--530}{10.1016/j.aim.2013.08.015}. \MR{3107519}

\bibitem[KS18]{KS18}
J.~Kaad and W.~D. van Suijlekom,
  \doilinktitle{\emph{{\noopsort{KS16}{R}iemannian submersions and
  factorization of {D}irac operators}}}{10.4171/JNCG/299}, \doilinkjournal{J.
  Noncommut. Geom.}{10.4171/JNCG/299} \doilinkvynp{ \textbf{12} (2018),
  1133--1159}{10.4171/JNCG/299}.

\bibitem[KS19]{KS19}
\bysame, \doilinktitle{\emph{{\noopsort{KS17a}On a theorem of Kucerovsky for
  half-closed chains}}}{10.7900/jot.2018mar07.2208}, \doilinkjournal{J.
  Operator Theory}{10.7900/jot.2018mar07.2208} \doilinkvynp{ \textbf{82}
  (2019), no.~1, 115--145}{10.7900/jot.2018mar07.2208}.

\bibitem[KS20]{KS20}
\bysame, \doilinktitle{\emph{{Factorization of Dirac operators on
  almost-regular fibrations of spin$^c$
  manifolds}}}{10.25537/dm.2020v25.2049-2084}, \doilinkjournal{Doc.
  Math.}{10.25537/dm.2020v25.2049-2084} \doilinkvynp{ \textbf{25} (2020),
  2049--2084}{10.25537/dm.2020v25.2049-2084}.

\bibitem[Kuc97]{Kuc97}
D.~Kucerovsky, \doilinktitle{\emph{The {KK}-product of unbounded
  modules}}{10.1023/A:1007751017966},
  \doilinkjournal{K-Theory}{10.1023/A:1007751017966} \doilinkvynp{ \textbf{11}
  (1997), 17--34}{10.1023/A:1007751017966}.

\bibitem[Lan95]{Lance95}
E.~Lance, \emph{Hilbert {$C^{\ast} $}-modules: A toolkit for operator
  algebraists}, Lecture note series: London Mathematical Society, Cambridge
  University Press, 1995.

\bibitem[Les05]{Les05}
M.~Lesch, \emph{The uniqueness of the spectral flow on spaces of unbounded
  self-adjoint {F}redholm operators}, Spectral geometry of manifolds with
  boundary and decomposition of manifolds (B.~Booss-Bavnbek, G.~Grubb, and
  K.~P. Wojciechowski, eds.), Contemp. Math., vol. 366, Amer. Math. Soc., 2005,
  pp.~193--224.

\bibitem[LM19]{LM19}
M.~{Lesch} and B.~{Mesland}, \doilinktitle{\emph{{Sums of regular self-adjoint
  operators in Hilbert-{$C^*$}-modules}}}{10.1016/j.jmaa.2018.11.059},
  \doilinkjournal{J. Math. Anal. Appl.}{10.1016/j.jmaa.2018.11.059}
  \doilinkvynp{ \textbf{472} (2019), 947--980}{10.1016/j.jmaa.2018.11.059}.

\bibitem[Mes14]{Mes14}
B.~Mesland, \doilinktitle{\emph{{Unbounded bivariant K-theory and
  correspondences in noncommutative geometry}}}{10.1515/crelle-2012-0076},
  \doilinkjournal{J. Reine Angew. Math.}{10.1515/crelle-2012-0076}
  \doilinkvynp{ \textbf{691} (2014), 101--172}{10.1515/crelle-2012-0076}.

\bibitem[MR16]{MR16}
B.~Mesland and A.~Rennie, \doilinktitle{\emph{Nonunital spectral triples and
  metric completeness in unbounded {KK}-theory}}{10.1016/j.jfa.2016.08.004},
  \doilinkjournal{J. Funct. Anal.}{10.1016/j.jfa.2016.08.004} \doilinkvynp{
  \textbf{271} (2016), no.~9, 2460--2538}{10.1016/j.jfa.2016.08.004}.

\bibitem[Ska84]{Ska84}
G.~Skandalis, \doilinktitle{\emph{Some remarks on {K}asparov
  theory}}{10.1016/0022-1236(84)90081-8}, \doilinkjournal{J. Funct.
  Anal.}{10.1016/0022-1236(84)90081-8} \doilinkvynp{ \textbf{56} (1984), no.~3,
  337--347}{10.1016/0022-1236(84)90081-8}.

\end{thebibliography}

\providecommand{\noopsort}[1]{}
\providecommand{\bysame}{\leavevmode\hbox to3em{\hrulefill}\thinspace}
\providecommand{\MR}{\relax\ifhmode\unskip\space\fi MR }
\providecommand{\MRhref}[2]{%
  \href{http://www.ams.org/mathscinet-getitem?mr=#1}{#2}
}
\providecommand{\href}[2]{#2}
\providecommand{\doilinktitle}[2]{#1}
\providecommand{\doilinkjournal}[2]{\href{https://doi.org/#2}{#1}}
\providecommand{\doilinkvynp}[2]{\href{https://doi.org/#2}{#1}}
\providecommand{\eprint}[2]{#1:\href{https://arxiv.org/abs/#2}{#2}}

\end{document}